\documentclass[11pt, a4paper]{amsart}
\usepackage{amsmath}%
\usepackage{amsfonts}%
\usepackage{amssymb}%
\usepackage{graphicx}
\usepackage{amsthm}
\usepackage{amssymb}
\usepackage{tikz}
\usepackage{tikz-cd}
\usetikzlibrary{matrix,arrows}
\usepackage[english, ngerman]{babel}
\usepackage[utf8]{inputenc}
\usepackage{hyperref}
\usepackage{cite}
\usepackage{footnote}
\usepackage{stmaryrd}
\usepackage[a4paper,margin=3cm]{geometry}
\usepackage{verbatim}
\usepackage{enumerate}

\newtheorem{thm}{Theorem}
\theoremstyle{plain}

\newtheorem{conjecture}{Conjecture}
\newtheorem{cor}[thm]{Corollary}

\newtheorem{lem}[thm]{Lemma}

\newtheorem{prop}[thm]{Proposition}

\theoremstyle{remark}

\newtheorem*{acknowledgements}{Acknowledgements}

\renewcommand{\phi}{\varphi}

\newcommand{\BC}{{\mathbb{C}}}

\newcommand{\BH}{{\mathbb{H}}}

\newcommand{\BN}{{\mathbb{N}}}

\newcommand{\BR}{{\mathbb{R}}}

\newcommand{\BZ}{{\mathbb{Z}}}

\newcommand{\Fa}{{\mathfrak{a}}}

\newcommand{\Fc}{{\mathfrak{c}}}

\newcommand{\CE}{{\mathcal E}}

\newcommand{\CH}{{\mathcal H}}
\newcommand{\CI}{{\mathcal I}}

\newcommand{\CM}{{\mathcal M}}

\newcommand{\CS}{{\mathcal S}}

\newcommand{\supp}{\mathop{\rm Supp}\nolimits}

\newcommand{\nin}{\notin}

\renewcommand{\mod}{\mathop{\rm mod}\nolimits}
\newcommand{\sign}{\mathop{\rm sign}\nolimits}

\newcommand{\SL}[1]{\mathop{\rm SL}_{#1} \nolimits}

\newcommand{\artanh}{\mathop{\rm artanh}\nolimits}

\renewcommand{\Re}{\mathop{\rm Re}\nolimits}
\renewcommand{\Im}{\mathop{\rm Im}\nolimits}

\newcommand{\quotient}[2]{
        \mathchoice
            {
                \text{\raise1ex\hbox{$#1$}\Big/\lower1ex\hbox{$#2$}}%
            }
            {
                #1\,/\,#2
            }
            {
                #1\,/\,#2
            }
            {
                #1\,/\,#2
            }
    }
		
\newcommand{\lquotient}[2]{
        \mathchoice
            {
                \text{\lower1ex\hbox{$#1$}\Big \backslash \raise01ex\hbox{$#2$}}%
            }
            {
                #1\,\backslash\,#2
            }
            {
                #1\,\backslash\,#2
            }
            {
                #1\,\backslash\,#2
            }
    }

\newcommand{\rquotient}[2]{
        \mathchoice
            {
                \text{\raise01ex\hbox{$#1$}\Big/\lower1ex\hbox{$#2$}}%
            }
						{
                #1\,/\,#2
            }
            {
                #1\,/\,#2
            }
            {
                #1\,/\,#2
            }
    }
		
\newcommand{\lrquotient}[3]{
        \mathchoice
            {
                \text{\lower1ex\hbox{$#1$}\Big \backslash \raise01ex\hbox{$#2$}\Big/\lower1ex\hbox{$#3$}}%
            }
            {
                #1\,\backslash\,#2\,/\,#3
            }
            {
                #1\,\backslash\,#2\,/\,#3
            }
            {
                #1\,\backslash\,#2\,/\,#3
            }
    }

\newcommand{\sm}{\left(\begin{smallmatrix}}
\newcommand{\esm}{\end{smallmatrix}\right)}
\newcommand{\bpm}{\begin{pmatrix}}
\newcommand{\ebpm}{\end{pmatrix}}

\newcommand{\one}{{\rm 1\mskip-4mu l}}

\numberwithin{equation}{section}
\begin{document}
\selectlanguage{english}

\bibliographystyle{plain}

\title[On a Twisted Version of Linnik and Selberg's Conjecture]{On a Twisted Version of Linnik and Selberg's Conjecture on Sums of Kloosterman Sums}
\author{Raphael S. Steiner}
\address{Department of Mathematics, University of Bristol, Bristol BS8 1TW, UK}%
\email{raphael.steiner@bristol.ac.uk}%




\begin{abstract} We generalise the work of Sarnak--Tsimerman to twisted sums of Kloosterman sums and thus give evidence towards the twisted Linnik--Selberg Conjecture. 
\end{abstract}
\maketitle


\section{Introduction}

The study of Kloosterman sums
$$
S(m,n;c)=\sum_{\substack{a \mod(c) \\ (a,c)=1}} e \left( \frac{ma+n\overline{a}}{c} \right), \text{ where } e(z)=e^{2\pi i z} \text{ and } a\overline{a} \equiv 1 \mod(c),
$$
is interesting for a variety of reasons. One of these reasons is their connection to the spectral theory of automorphic forms. In particular the sign changes of $S(m,n;c)$, for $c$ varying in the arithmetic progression $c \equiv 0 \mod(s)$, are related to the Selberg conjecture about the smallest positive eigenvalue of the Laplacian on the space $\lquotient{\Gamma_0(s)}{\BH}$. Concretely we have that the smallest positive eigenvalue $\lambda_1^s \ge \frac{1}{4}$ if and only if the following conjecture holds (see \cite[Theorem 16.9]{ANT}).

\begin{conjecture}[Smooth Linnik in AP] Let $m,s \in \BN$, $g \in C^3(\BR^+,\BR^+_0)$ a compactly supported bump function with $|g^{(a)}|\le 1$ for $a=0,1,2,3$, and $C \ge 1$. Then we have for every $\epsilon>0$
$$
\sum_{\substack{c \equiv 0 \mod(s)}} \frac{1}{c}S(m,m;c) g\left( \frac{C}{c} \right) \ll_{\epsilon, m, s} C^{\epsilon}.
$$
\end{conjecture}
In this paper however, we are interested in the sharp cut-off variant of the above conjecture. 
The first non-trivial progress towards this conjecture was made by Kuznetsov \cite{Kuz1}, who managed to prove
\begin{equation}
\sum_{c \le C} \frac{1}{c}S(m,n;c) \ll_{m,n} C^{\frac{1}{6}}\log(2C)^{\frac{1}{3}},
\label{eq:Kuzbound}
\end{equation}
by exploiting the Kuznetsov trace formula (see Proposition \ref{prop:kuznetsov}), which was established in the same paper. The bound \eqref{eq:Kuzbound} is still the best known bound to date and the Kuznetsov trace formula has become a very powerful tool in a variety of contexts.

In their paper \cite{SarTsim} Sarnak--Tsimerman have made the dependence on $m,n$ in \eqref{eq:Kuzbound} explicit and moreover achieved a non-trivial bound in the harder `Selberg' range ($C \le \sqrt{|mn|}$). Their result has further been generalised to the arithmetic progressions $c \equiv 0 \mod (s)$ by Ganguly--Sengupta \cite{APKloosterman}, and to $c \equiv a \mod(r)$ with $(a,r)=1$ by Blomer--Mili\'cevi\'c \cite{BlomerKloos}. Recently Kiral--Young \cite{KiralYoungfifth} have indicated a simple approach which allows one to incorporate both congruence conditions $c \equiv 0 \mod(s)$ and $c \equiv a \mod(r)$ simultaneously (assuming $(r,as)=1$).

Motivated by an application to the efficiency of a certain universal set of quantum gates, Browning--Kumaraswamy--Steiner \cite{S3covexp} have proposed the following twisted version of the Linnik--Selberg conjecture.

\begin{conjecture}[Twisted Linnik--Selberg] 
Let $B, C\geq 1$
and let  $m,n\in \BZ$ be non-zero. 
Let $s\in \BN$ and let $a\in \BZ/s\BZ$. 
Then for any $\alpha\in [-B,B]$ we have
$$
\sum_{\substack{c \equiv a \mod{(s)}\\  c \leq C}}  \frac{1}{c}S(m ,n;c)e\left(\frac{2\sqrt{mn}}{c} \alpha\right) 
 \ll_{\epsilon,s,B} (|mn| C)^{\epsilon},
 $$
for any $\epsilon>0$.
\label{conj:2}
\end{conjecture}
In this paper we are concerned with establishing some progress towards this conjecture. Before we state our results we shall introduce some simplifying notation: $F \lesssim G$ means $|F| \le K_{\epsilon} (Cmns(1+|\alpha|))^{\epsilon} G $ for some positive constant $K_{\epsilon}$, depending on $\epsilon$, and every $\epsilon>0$.
\begin{thm} Let $C \ge 1$, $\alpha \in \BR$, $s\in \BN$ and $m,n\in \BZ$ with $mn>0$, $s\ll \min \{(mn)^{\frac{1}{4}},C^{\frac{1}{2}}\}$, and $(m,n,s)=1$. Then we have
\begin{multline*}
\sum_{\substack{c \le C \\ c \equiv 0 \mod(s)}} \frac{1}{c}S(m,n;c)e\left(\frac{2\sqrt{mn}}{c} \alpha\right) + 2 \pi \sum_{t_h \in i[0,\theta]} \frac{\sqrt{mn} \cdot \overline{\rho_h(m)}\rho_h(n)}{\cos(\pi |t_h|)} \int_{\frac{\sqrt{mn}}{C}}^{\infty} Y_{2|t_h|}(x)e^{i\alpha x}\frac{dx}{x} \\
 \lesssim \frac{C^{\frac{1}{6}}}{s^{\frac{1}{3}}}+(1+|\alpha|^{\frac{1}{3}})\frac{(mn)^{\frac{1}{6}}}{s^{\frac{2}{3}}}+\frac{m^{\frac{1}{4}}(m,s)^{\frac{1}{4}}+n^{\frac{1}{4}}(n,s)^{\frac{1}{4}}}{s^{\frac{1}{2}}} \\
+\min \left \{\frac{(mn)^{\frac{1}{8}+\frac{\theta}{2}}(mn,s)^{\frac{1}{8}}}{s^{\frac{1}{2}}} , \frac{(mn)^{\frac{1}{4}}(mn,s)^{\frac{1}{4}}}{s} \right \} ,
\end{multline*}
where $Y_t$ is the Bessel function of the second kind of order $t$, $\theta$ is the best known progress towards the Ramanujan--Selberg conjecture, and the summation $t_h$ is over all exceptional eigenfunctions $h$ with eigenvalue $\frac{1}{4}+t_h^2$ of the Laplacian for the manifold $\lquotient{\Gamma_0(s)}{\BH}$, where $\rho_h(n)$ denotes its $n$-th $L^2$-normalised Fourier coefficient.
\label{thm:1}
\end{thm}

A few remarks are in order about this theorem. First we should remark that one has $\theta\le\frac{7}{64}$ by the work of Kim--Sarnak \cite{KimSar}. Next we observe the appearance of a main term, which is contrary to \cite{APKloosterman}. Indeed, the latter has an erroneous treatment of the exceptional spectrum\footnote{The compact domain to which they apply the mean value theorem of calculus varies and this may not be circumvented, since if the exceptional spectrum is non-empty then the function they consider has a pole at $0$.}. One may further analyse the main term by making use of asymptotics of the Bessel function of the second kind $Y_t(y)$ for $y \to 0$. However the reader familiar with Bessel functions may know that these asymptotics behave quite differently for $t=0$ and $t>0$ and therefore it would generate uniformity issues in the parameter $s$. One may also bound the main term altogether. In this case one gets the following corollary.

\begin{cor} Assume the same assumptions as in Theorem \ref{thm:1}. Then
\begin{multline*}
\sum_{\substack{c \le C \\ c \equiv 0 \mod(s)}} \frac{1}{c}S(m,n;c)e\left(\frac{2\sqrt{mn}}{c} \alpha\right) \\
\lesssim \frac{C^{\frac{1}{6}}}{s^{\frac{1}{3}}}+C^{2\theta}+(1+|\alpha|^{\frac{1}{3}})\frac{(mn)^{\frac{1}{6}}}{s^{\frac{2}{3}}}+\frac{m^{\frac{1}{4}}(m,s)^{\frac{1}{4}}+n^{\frac{1}{4}}(n,s)^{\frac{1}{4}}}{s^{\frac{1}{2}}} \\
+\min \left \{\frac{(mn)^{\frac{1}{8}+\frac{\theta}{2}}(mn,s)^{\frac{1}{8}}}{s^{\frac{1}{2}}} , \frac{(mn)^{\frac{1}{4}}(mn,s)^{\frac{1}{4}}}{s} \right \}.
\end{multline*}
\label{cor:2}
\end{cor}

As far as the restrictions go in Theorem \ref{thm:1}, they are not very limiting. Indeed if $s \ge C^{\frac{1}{2}}$, then the Weil bound, which gives the bound $s^{-1+\epsilon}C^{\frac{1}{2}+\epsilon}$, is more than sufficient, and if $(mn)^{\frac{1}{4}} \le s \le C^{\frac{1}{2}}$ then one is automatically in the easier Linnik range and for instance the holomorphic contribution is negligible. One may also consider $mn<0$, which would lead one to analyse different Bessel transforms, or incorporate the further restriction $c \equiv a \mod(r)$ with $(a,r)=1$. However, for the latter, an analogue to Proposition \ref{prop:Maassavg} for the group $\Gamma_0(s)\cap \Gamma_1(r)$ has to be derived. In fact the associated Kloosterman sums for this group admit further cancellation, thus leading to stronger results in terms of the parameter $r$. Investigations of this sort shall be considered by the author in future work.

For $|\alpha|<1$ one may improve Theorem \ref{thm:1} slightly, thereby recovering the results of \cite{SarTsim} and \cite{APKloosterman}.
\begin{thm} Let $C \ge 1$, $\alpha \in \BR$ with $|\alpha|<1$, $s\in \BN$ and $m,n\in \BZ$ with $mn>0$, $s\ll \min \{(mn)^{\frac{1}{4}},C^{\frac{1}{2}}\}$, and $(m,n,s)=1$. Then we have
\begin{multline*}
\sum_{\substack{c \le C \\ c \equiv 0 \mod(s)}}  \frac{1}{c}S(m,n;c)e\left(\frac{2\sqrt{mn}}{c} \alpha\right) + 2 \pi \sum_{t_h \in i[0,\theta]} \frac{\sqrt{mn} \cdot \overline{\rho_h(m)}\rho_h(n)}{\cos(\pi |t_h|)} \int_{\frac{\sqrt{mn}}{C}}^{\infty} Y_{2|t_h|}(x)e^{i\alpha x}\frac{dx}{x} \\
\lesssim (1-|\alpha|)^{-\frac{1}{2}-\epsilon} \Biggl( \frac{C^{\frac{1}{6}}}{s^{\frac{1}{3}}}+\frac{m^{\frac{1}{8}}(m,s)^{\frac{1}{8}}+n^{\frac{1}{8}}(n,s)^{\frac{1}{8}}}{s^{\frac{1}{4}}}\min \left \{ (mn)^{\frac{\theta}{2}} , \frac{m^{\frac{1}{8}}(m,s)^{\frac{1}{8}}+n^{\frac{1}{8}}(n,s)^{\frac{1}{8}}}{s^{\frac{1}{4}}} \right \} \\
+ \frac{(mn)^{\frac{1}{6}}}{s^{\frac{2}{3}}}+ \min \left \{\frac{(mn)^{\frac{1}{16}+\frac{3\theta}{4}}(mn,s)^{\frac{1}{16}}}{s^{\frac{1}{4}}} , \frac{(mn)^{\frac{1}{4}}(mn,s)^{\frac{1}{4}}}{s} \right \} \Biggr)
\end{multline*}
and
\begin{multline*}
\sum_{\substack{c \le C \\ c \equiv 0 \mod(s)}}  \frac{1}{c}S(m,n;c)e\left(\frac{2\sqrt{mn}}{c} \alpha\right) \\
\lesssim (1-|\alpha|)^{-\frac{1}{2}-\epsilon} \Biggl( \frac{C^{\frac{1}{6}}}{s^{\frac{1}{3}}}+\frac{m^{\frac{1}{8}}(m,s)^{\frac{1}{8}}+n^{\frac{1}{8}}(n,s)^{\frac{1}{8}}}{s^{\frac{1}{4}}}\min \left \{ (mn)^{\frac{\theta}{2}} , \frac{m^{\frac{1}{8}}(m,s)^{\frac{1}{8}}+n^{\frac{1}{8}}(n,s)^{\frac{1}{8}}}{s^{\frac{1}{4}}} \right \} \\
+ \frac{(mn)^{\frac{1}{6}}}{s^{\frac{2}{3}}}+ \min \left \{\frac{(mn)^{\frac{1}{16}+\frac{3\theta}{4}}(mn,s)^{\frac{1}{16}}}{s^{\frac{1}{4}}} , \frac{(mn)^{\frac{1}{4}}(mn,s)^{\frac{1}{4}}}{s} \right \} \Biggr) + C^{2\theta}.
\end{multline*}
\label{thm:2}
\end{thm}

The main goal in \cite{S3covexp} was to show that it is possible to improve Sardari's work on covering exponents for $S^3$ \cite{NaserCovExp} under the assumption that Conjecture \ref{conj:2} holds. It is unfortunate that the derived upper bounds in Theorem \ref{thm:1} and \ref{thm:2} are not strong enough to offer any unconditional improvement. The reason behind this is that in the application one is very deep in the Selberg range, for which the trivial bound is still the best known bound. Discussions on exactly why the Selberg range poses great difficulties can be found in \cite{SarTsim}.

Finally, we would like to point out a little gem that is hidden inside Theorem \ref{thm:1}.
\begin{cor} Let $C \in \BR^+$ and $Q(T)=mT^2+lT+n \in \BZ[T]$ with $mn>0$. Then we have
$$
\sum_{c \le C} \frac{1}{c} \sum_{\substack{a \mod (c)\\ (a,c)=1}} e\left( \frac{Q(a)\overline{a}}{c} \right) \ll_{\epsilon} C^{\frac{1}{6}+\epsilon}+\max\{|l|,|m|,|n|\}^{\frac{23}{64}+\epsilon}.
$$
\end{cor}
This has as a consequence that either there is cancellation in the sign or very often the inner exponential sum is much smaller than $\sqrt{c}$.

\begin{acknowledgements} I would like to thank my supervisors Andrew Booker and Tim Browning for the detailed read-throughs and comments on earlier versions of this paper as well as Mehmet Kiral, Matt Young and Nick Andersen for discussions on this and related topics.

This material is partially based upon work supported by the National Science Foundation under Grant No. DMS-1440140 while the author was in residence at the Mathematical Sciences Research Institute in Berkeley, California, during the Spring 2017 semester.
\end{acknowledgements}

\section{Holomorphic and Maass Forms}
In this section we set up some notation and recall necessary facts about holomorphic and Maass forms.

Let $\mathbb{H}$ be the upper half-plane and let $\SL{2}(\BR)$ act on it by M\"obius transformations:
$$
\gamma\cdot z =\gamma z = \frac{az+b}{cz+d}, \quad j(\gamma,z)=cz+d, \text{ where }  \gamma=\begin{pmatrix}a & b \\ c & d\end{pmatrix} \in \SL{2}(\BR).
$$
We consider the following congruence subgroup
$$
\Gamma_0(s)= \left\{ \begin{pmatrix}a & b \\ c & d\end{pmatrix} \in \SL{2}(\BZ) \bigg| c \equiv 0\, \mod s \right \}.
$$
For a given cusp $\Fa$ of $\Gamma_0(s)$ we fix a matrix $\sigma_{\Fa} \in \SL{2}(\BR)$, such that $\sigma_{\Fa}\infty= \Fa$ and if $\Gamma_{\Fa}$ denotes the stabilizer of $\Fa$ then $ \sigma_{\Fa}^{-1}\Gamma_{\Fa}\sigma_{\Fa}= \Gamma_{\infty}$, where $\Gamma_{\infty}=\{ \pm T^n | n \in \BZ  \}$ is the stabilizer at $\infty$ and $T=\bigl(\begin{smallmatrix} 1&1\\0&1 \end{smallmatrix}\bigr)$. Such a matrix is called a scaling matrix for the cusp $\Fa$.\\

The space of cuspidal Maass forms consists of the real-analytic square integrable eigenfunctions of the Laplacian on the space $L^2(\lquotient{\Gamma_0(s)}{\BH})$ with respect to the inner product
\begin{equation}
\langle h_1 , h_2 \rangle =\int_{\lquotient{\Gamma_0(s)}{\BH}} h_1(z) \overline{h_2(z)} \frac{dxdy}{y^2}.
\label{eq:Maassinner}
\end{equation}
Such a Maass form $h$ possesses a Fourier expansion of the shape
$$
h(z)=\sum_{\substack{n \in \BZ \\ n \neq 0}} \rho_h(n) W_{0,it_h}(4 \pi |n| y)e(nx),
$$
where $W_{a,b}$ is the Whittaker function, $z=x+iy$, and $\frac{1}{4}+t_h^2 \ (t_h \in [0,\infty)\cup i[0,1/2])$ is the eigenvalue with respect to the Laplacian. A theory of Hecke operators as well as Atkin--Lehner theory can be developed for this space. In particular for a newform $h$ we have
$$
\sqrt{n} \rho_h(n)=\lambda_h(n)\rho_h(1), \quad \forall n \in \BN,
$$
where $\lambda_h(n)$ is the eigenvalue with respect to the $n$-th Hecke operator, which furthermore satisfies $\lambda_h(n) \ll_{\epsilon} n^{\theta+\epsilon}$, where $\theta=\frac{7}{64}$ is admissible by the work of Kim and Sarnak \cite{KimSar}.

We shall require a special basis of this space which has been worked out in \cite{BlomONB}\footnote{Corrections can be found at http://www.uni-math.gwdg.de/blomer/corrections.pdf}. For a Maass newform of level $r | s$ define the arithmetic functions

$$
r_h(c)=\sum_{b|c}\frac{\mu(b)\lambda_h(b)^2}{b}\left( \sum_{d|c} \frac{\chi_0(d)}{d} \right)^{-2}\!\!\!\!\!\!\!\!, \quad A(c)=\sum_{b|c} \frac{\mu(b)\chi_0(b)^2}{b^2}, \quad B(c)=\sum_{b|c}\frac{\mu(b)^2\chi_0(b)}{b},
$$
where $\chi_0$ is the trivial character modulo $r$, and the multiplicative function $\mu_h(c)$ is defined by the equation
$$
\left(\sum_{c \ge 1} \frac{\lambda_h(c)}{c^s} \right)^{-1}= \sum_{c \ge 1} \frac{\mu_h(c)}{c^s} .
$$
For $l|d$ define
$$
\xi_d'(l)=\frac{\mu(d/l)\lambda_h(d/l)}{r_h(d)^{\frac{1}{2}}(d/l)^{\frac{1}{2}}B(d/l)}, \quad \xi_d''(l)= \frac{\mu_h(d/l)}{r_h(d)^{\frac{1}{2}}(d/l)^{\frac{1}{2}}A(d)^{\frac{1}{2}}}.
$$
Write $d=d_1d_2$ with $d_1$ square-free and $d_2$ square-full and $(d_1,d_2)=1$. Then for $l|d$ define
\begin{equation}
\xi_d(l)=\xi_{d_1}'((d_1,l))\xi_{d_2}''((d_2,l)) \ll_{\epsilon} d^{\epsilon}.
\label{eq:xifunc}
\end{equation}
Then an orthonormal basis of Maass forms of level $s$ is given by
\begin{equation}
\bigcup_{\substack{r|s}}\bigcup_{\substack{h \text{ new} \\ \text{of level }r}}\left\{ h^d(z)= \sum_{l|d}\xi_d(l)h(lz) \Bigg | d | \tfrac{s}{r} \right \}.
\label{eq:MaassONB}
\end{equation}
We furthermore need a bound on the size of the Fourier coefficient of an element of the above basis. We have
\begin{equation}\begin{aligned}
\sqrt{n} \rho_{h^d}(n) &= \sum_{l|(d,n)} \sqrt{l} \xi_d(l)  \lambda_h\left( \frac{n}{l} \right) \rho_h(1) \\ 
& \ll_{\epsilon}  (ns)^{\epsilon} n^{\theta} |\rho_h(1)| \sum_{l|(d,n)} l^{\frac{1}{2}-\theta} \\
&  \ll_{\epsilon} (ns)^{\epsilon} n^{\theta} \left( \frac{s}{r} \right)^{\frac{1}{2}} |\rho_h(1)|,
\label{eq:Maassfourierbound}
\end{aligned}\end{equation}
where we have made use of \eqref{eq:xifunc} and $\lambda_h(n) \ll_{\epsilon} n^{\theta +\epsilon}$. Since $h$ is new of level $r$, but normalised with respect to the inner product of level $s$ \eqref{eq:Maassinner} we further have
\begin{equation}
|\rho_h(1)| \ll_{\epsilon} (s(1+|t_h|))^{\epsilon} \left(\frac{\cosh(\pi t_h)}{s} \right)^{\frac{1}{2}},
\label{eq:Maassnorm}
\end{equation}
due to Hoffstein and Lockhart \cite{HoffLock}.

Other Maass forms which are important in our discussion are the Eisenstein series associated to a cusp $\Fc$. They are defined for $\Re(\tau) > 1$ as
$$
E_{\Fc}(z,\tau) = \sum_{\gamma \in \lquotient{\Gamma_{\infty}}{\sigma_{\Fc}^{-1}\Gamma_0(s)}} \Im(\gamma z)^{\tau}
$$
and admit a meromorphic extension to the whole complex plane. They also admit a Fourier expansion of the same shape, which at the point $\tau=\frac{1}{2}+it$ we write as
$$
E_{\Fc}(z,\tfrac{1}{2}+it) = \phi_{\Fc}(0,t;z) + \sum_{n \neq 0} \phi_{\Fc}(n,t) W_{0,it}(4\pi|n|y)e(nx).
$$

For holomorphic forms the situation is quite analogous. A holomorphic cusp form of weight $k \in \BN$ of level $s$ is a holomorphic function $h: \BH \to \BC$ that satisfies $j(\gamma,z)^{-k}h(\gamma z) = h(z)$ for all $\gamma \in \Gamma_0(s)$ and is square integrable with respect to the inner product
\begin{equation}
\langle h_1 , h_2 \rangle =\int_{\lquotient{\Gamma_0(s)}{\BH}} h_1(z) \overline{h_2(z)} y^k \frac{dxdy}{y^2}.
\label{eq:holoinner}
\end{equation}
They admit a Fourier expansion of the shape
$$
h(z)= \sum_{n \ge 1} \psi_h(n) e(nz)
$$
and there is a theory of Hecke and Atkin--Lehner operators. For $h$ a newform we have
$$
\psi_h(n)=\lambda_h(n) \psi_h(1),
$$
where $\lambda_h(n)$ is the eigenvalue of the $n$-th Hecke operator, which furthermore satisfies the bound $\lambda_h(n) \ll_{\epsilon} n^{\frac{k-1}{2}+\epsilon}$ due to Deligne \cite{Deligne71}, \cite{Deligne74} and Deligne-Serre \cite{DeligneSerre75}. Analogous to the Maass case we have a nice orthonormal basis of the space $S_k(s)$ of holomorphic cusp forms of level $s$ and weight $k$:
\begin{equation}
\bigcup_{\substack{r|s}}\bigcup_{\substack{h \text{ new} \\ \text{of level }r}}\left\{ h^d(z)= \sum_{l|d}\xi_d(l)l^{\frac{k}{2}}h(lz) \Bigg | d | \tfrac{s}{r} \right \}.
\label{eq:holONB}
\end{equation}
We furthermore need a bound on the size of the Fourier coefficients of an element of the above basis. We have
\begin{equation}\begin{aligned}
\psi_{h^d}(n) &= \sum_{l|(d,n)} \xi_d(l) l^{\frac{k}{2}}  \lambda_h\left( \frac{n}{l} \right) \psi_h(1) \\ 
& \ll_{\epsilon}  (ns)^{\epsilon} n^{\frac{k-1}{2}} |\psi_h(1)| \sum_{l|(d,n)} l^{\frac{1}{2}} \\
&  \ll_{\epsilon} (ns)^{\epsilon} n^{\frac{k-1}{2}} \left( \frac{s}{r} \right)^{\frac{1}{2}} |\psi_h(1)|,
\label{eq:fourierbound}
\end{aligned}\end{equation}
where we have made use of the Deligne bound as well as \eqref{eq:xifunc}. We further have the bound
\begin{equation}
|\psi_h(1)| \ll_{\epsilon} \frac{(4\pi)^{\frac{k-1}{2}}}{s^{\frac{1}{2}} \Gamma(k)^{\frac{1}{2}}}(ks)^{\epsilon},
\label{eq:firstcoef}
\end{equation}
when $h$ is new of level $r$, but normalised with respect to \eqref{eq:holoinner}; see for example \cite[pp. 41,42]{ANTALfunc}.

\section{Proof of the Theorem}
\label{sec:proof}

We shall prove a dyadic version of Theorem \ref{thm:1} from which we shall then deduce Theorem~\ref{thm:1}.

\begin{thm}
Let $\alpha \in \BR$, $s\in \BN$ and $m,n\in \BZ$ with $mn>0$ and $(m,n,s)=1$. Assume $s \ll \min \{(mn)^{\frac{1}{4}}, C^{\frac{1}{2}} \}$. Then we have
\begin{multline*}
\sum_{\substack{C \le c < 2C \\ c \equiv 0 \mod(s)}}  \frac{1}{c}S(m,n;c)e\left(\frac{2\sqrt{mn}}{c} \alpha\right) + 2 \pi \sum_{t_h \in i[0,\theta]} \frac{\sqrt{mn} \cdot \overline{\rho_h(m)}\rho_h(n)}{\cos(\pi |t_h|)} \int_{\frac{X}{2}}^X Y_{2|t_h|}(x)e^{i\alpha x}\frac{dx}{x} \\
 \lesssim \frac{C^{\frac{1}{6}}}{s^{\frac{1}{3}}}+(1+|\alpha|)\frac{(mn)^{\frac{1}{2}}}{C}+\frac{m^{\frac{1}{4}}(m,s)^{\frac{1}{4}}+n^{\frac{1}{4}}(n,s)^{\frac{1}{4}}}{s^{\frac{1}{2}}} \\
+\min \left \{ \frac{(mn)^{\frac{\theta}{2}+\frac{1}{8}}(mn,s)^{\frac{1}{8}}}{s^{\frac{1}{2}}}, \frac{(mn)^{\frac{1}{4}}(mn,s)^{\frac{1}{4}}}{s} \right\}.
\end{multline*}
For $|\alpha|<1 $ we can do slightly better:
\begin{multline*}
\sum_{\substack{C \le c < 2C \\ c \equiv 0 \mod(s)}} \frac{1}{c}S(m,n;c)e\left(\frac{2\sqrt{mn}}{c} \alpha\right) +  2 \pi \sum_{t_h \in i[0,\theta]} \frac{\sqrt{mn} \cdot \overline{\rho_h(m)}\rho_h(n)}{\cos(\pi |t_h|)} \int_{\frac{X}{2}}^X Y_{2|t_h|}(x)e^{i\alpha x}\frac{dx}{x} \\
\lesssim (1-|\alpha|)^{-\frac{1}{2}-\epsilon} \Biggl( \frac{C^{\frac{1}{6}}}{s^{\frac{1}{3}}}  +\frac{m^{\frac{1}{8}}(m,s)^{\frac{1}{8}}+n^{\frac{1}{8}}(n,s)^{\frac{1}{8}}}{s^{\frac{1}{4}}} \min \left \{ (mn)^{\frac{\theta}{2}}  , \frac{m^{\frac{1}{8}}(m,s)^{\frac{1}{8}}+n^{\frac{1}{8}}(n,s)^{\frac{1}{8}}}{s^{\frac{1}{4}}} \right \} \\
+\frac{(mn)^{\frac{1}{2}}}{C}+\min \left \{ \frac{(mn)^{\frac{3\theta}{4}+\frac{1}{16}}(mn,s)^{\frac{1}{16}}}{s^{\frac{1}{4}}}, \frac{(mn)^{\frac{1}{4}}(mn,s)^{\frac{1}{4}}}{s} \right\} \Biggr).
\end{multline*}
\label{thm:dyadic}
\end{thm}
We follow the argument in \cite{SarTsim} and \cite{APKloosterman}, and replace the sharp cut off with a smooth cut off and then use Kuznetsov's trace formula. We shall require the following version of the Kuznetsov trace formula.

\begin{prop}[Kuznetsov's trace formula] Let $s \in \BN$ and $m,n \in \BZ$ be two integers with $mn>0$. Then for any $C^3$-class function $f$ with compact support in $]0,\infty)$ one has
$$\begin{aligned}
\sum_{c \equiv 0 \mod(s)} \frac{1}{c}S(m,n; c) f \left( \frac{4 \pi \sqrt{mn}}{c}  \right) =& \CH^{s}(m,n;f)+\CM^{s}(m,n;f) +\CE^{s}(m,n;f),
\end{aligned}$$
where
$$\begin{aligned}
\CH^{s}(m,n;f) =&\frac{1}{\pi} \sum_{\substack{k \equiv 0 \mod(2)\\k>0}} \sum_{\substack{\{h_{j,k}\}_j \text{ ONB} \\ \text{of } S_k(s) }} \frac{i^k \Gamma(k)}{(4 \pi \sqrt{mn})^{k-1}} \overline{\psi_{h_{j,k}}(m)}\psi_{h_{j,k}}(n) \widetilde{f}(k-1) \\
\CM^{s}(m,n;f)=&4 \pi \sum_{h} \frac{\sqrt{mn}}{\cosh \pi t_h}\overline{\rho_{h}(m)}\rho_{h}(n) \widehat{f}(t_h), \\
\CE^{s}(m,n;f)=& \sum_{\Fc \text{ cusp}} \int_{-\infty}^{\infty} \frac{\sqrt{mn}}{\cosh(\pi t)} \overline{\phi_{\Fc}(m,t)}\phi_{\Fc}(n,t) \widehat{f}(t)dt.
\end{aligned}$$
Here $\sum_h$ is a sum over an orthonormal basis of Maass forms with respect to the group $\Gamma_0(s)$ and the Bessel transforms are given by
$$\begin{aligned}
\widetilde{f}(t) &= \int_{0}^{\infty} J_t(y) f(y) \frac{dy}{y}, \\
\widehat{f}(t) &= \frac{i}{\sinh \pi t} \int_0^{\infty} \frac{J_{2it}(x)-J_{-2it}(x)}{2}f(x) \frac{dx}{x},
\end{aligned}$$
where $J_t(y)$
 is the Bessel function of the first kind of order $t$.
\label{prop:kuznetsov}
\end{prop}
\begin{proof} See \cite{Proskurin} or \cite{generalkuznetsov}.
\end{proof}

From now on let $f(x)=e^{i\alpha x}g(x)$ with $g$ smooth real-valued bump function satisfying the following properties
\begin{enumerate}[(i)]
	\item $g(x)=1$ for $\frac{2 \pi \sqrt{mn}}{C} \le x \le \frac{4 \pi \sqrt{mn}}{C}$,
	\item $g(x)=0$ for $x \le \frac{2 \pi \sqrt{mn}}{C+T}$ and $x \ge \frac{4 \pi \sqrt{mn}}{C-T}$,
	\item $\|g'\|_1 \ll 1$ and $\|g''\|_1 \ll \frac{C}{X \cdot T}$,
\end{enumerate}
where
$$
X=\frac{4\pi\sqrt{mn}}{C} \text{ and } 1 \le T \le \frac{C}{2}
$$
is a parameter to be chosen at a later point. Note that we have $\supp g \subseteq [\frac{X}{3},2X]$.\\

We now wish to compare the smooth sum
\begin{equation}
\sum_{\substack{c \equiv 0 \mod(s)}}  \frac{1}{c}S(m,n;c)f\left(\frac{4\pi\sqrt{mn}}{c} \alpha\right)
\label{eq:smooth}
\end{equation}
with the sharp cut off in Theorem \ref{thm:dyadic}. By making use of the Weil bound for the Kloosterman sum we find that their difference is bounded by
\begin{equation}\begin{aligned}
\sum_{\substack{C-T \le c \le C \text{ or}\\2C \le c \le 2C+2T, \\ c \equiv 0 \mod(s)}} \frac{1}{c} |S(m,n;c)|
& \le \sum_{\substack{C-T \le c \le C \text{ or}\\2C \le c \le 2C+2T \\ c \equiv 0 \mod(s)}} \frac{\tau(c)}{\sqrt{c}} (m,n,c)^{\frac{1}{2}} \\
& \le \frac{\tau(s)}{\sqrt{s}} \sum_{e|(m,n)} \sum_{\substack{\frac{C-T}{se} \le c' \le \frac{C}{se} \text{ or}\\\frac{2C}{se} \le c' \le \frac{2C+2T}{se}}} \frac{\tau(ec')}{\sqrt{ec'}} e^{\frac{1}{2}} \\
& \lesssim \frac{1}{\sqrt{s}} \sum_{e|(m,n)} \frac{\sqrt{se}}{\sqrt{C}} \left(1+\frac{T}{se} \right) \\
& \lesssim \frac{1}{\sqrt{C}}\left( (m,n)^{\frac{1}{2}}+ \frac{T}{s} \right).
\end{aligned}\label{eq:sharperror}
\end{equation}
Now we apply Kuznetsov (see Proposition \ref{prop:kuznetsov}) to the smooth sum \eqref{eq:smooth}. This leads to the expression
$$\begin{aligned}
\sum_{\substack{ c \equiv 0 \mod(s)}} \frac{1}{c}S(m,n;c)f\left(\frac{4\pi\sqrt{mn}}{c} \right) = \CH^{s}(m,n;f)+\CM^{s}(m,n;f)+\CE^{s}(m,n;f).\end{aligned}$$
We shall deal with each of these terms separately. In what follows we shall use many estimates on the Bessel transforms of $f$, which we shall summarise here, but postpone their proof until Section \ref{sec:est}.

\begin{lem}Let $f$ be as in the beginning of Section \ref{sec:proof}. Then we have
\begin{align}
\widehat{f}(t) , \widetilde{f}(t) &\ll  \frac{1+|\log(X)|+\log^{+}(|\alpha|)}{1+X^{\frac{1}{2}}+||\alpha|^2-1|^{\frac{1}{2}}X}, && \forall t \in \BR, \label{eq:trivial}\\
\widehat{f}(it) &= -\frac{1}{2}\int_{\frac{X}{2}}^{X} Y_{2t}(x)e^{i\alpha x}\frac{dx}{x} + O_{\epsilon, \delta}\left(1+ \frac{T}{C}X^{-2t-\epsilon} \right), && \forall 0 \le t\le \frac{1}{4}-\delta, \label{eq:exasymp}
\end{align}
where $\log^+(x)=\max\{0,\log(x)\}$. For $t\ge 8$ we have
\begin{align}
\int_0^{\frac{t}{2}} J_t(y) f(y) \frac{dy}{y} &\ll \one_{[2X/3, \infty)}(t) \cdot t^{-\frac{1}{2}} e^{-\frac{2}{5}t}, \label{eq:holosplit1} \\
\int_{\frac{t}{2}}^{t-t^{\frac{1}{3}}} J_t(y)f(y) \frac{dy}{y} &\ll
\one_{[\frac{1}{4},\infty)}(X)\one_{[X/3,4X]}(t) \cdot t^{-1}(\log(t))^{\frac{2}{3}}, \label{eq:holosplit2} \\
\int_{t-t^{\frac{1}{3}}}^{t+t^{\frac{1}{3}}} J_t(y) f(y) \frac{dy}{y} &\ll \one_{[\frac{1}{4},\infty)}(X)\one_{[3X/16,3X]}(t) \cdot t^{-1},  \label{eq:holosplit3}\\
\int_{t+t^{\frac{1}{3}}}^{\infty} J_t(y) f(y) \frac{dy}{y} &\ll  \one_{[\frac{1}{4},\infty)}(X)\one_{[0,3X/2]}(t) \cdot t^{-1} \min \left \{ 1+ |1-|\alpha||^{-\frac{1}{4}}, \left(  \frac{X}{t} \right)^{\frac{1}{2}} \right \}, \label{eq:holosplit4}
\end{align}
where $\one_{\CI}$ is the characteristic function of the interval $\CI$. Finally when $|t|\ge 1$ and either $|t| \nin \left[\tfrac{1}{12}||\alpha|^2-1|^{\frac{1}{2}}X,2||\alpha|^2-1|^{\frac{1}{2}}X\right]$ or $|\alpha| \le 1$ we have
\begin{align}
\widehat{f}(t) & \ll |t|^{-\frac{3}{2}} \left(1+ \min\left\{\left( \frac{X}{|t|} \right)^{\frac{1}{2}},||\alpha|^2-1|^{-1} \left( \frac{X}{|t|} \right)^{-\frac{3}{2}} \right\} \right) , \label{eq:nonholo4}\\
\widehat{f}(t) & \ll \frac{C}{T} |t|^{-\frac{5}{2}} \left(1+ \min\left\{\left( \frac{X}{|t|} \right)^{\frac{3}{2}},||\alpha|^2-1|^{-2} \left( \frac{X}{|t|} \right)^{-\frac{5}{2}} \right\} \right). \label{eq:nonholo5}
\end{align}
\label{lem:upperbnds}
\end{lem}
One should mention that similar estimates have been derived previously by Jutila \cite{JutilaConvolutionCusp}, for a slightly different class of functions and ranges.

\subsection{The Continuous Spectrum}

The goal of this section is to prove the following bound on the continuous contribution
\begin{equation}
\CE^{s}(m,n;f) \lesssim 1, 
\label{eq:cont}
\end{equation}
For this endeavour we need the following lemma.

\begin{lem} Let $s=s_{\star}s_{\square}^2$ with $s_{\star}$ square-free and let $m,n$ positive integers. We have
\begin{equation*}
\sum_{\Fc \text{ cusp}} \frac{\sqrt{mn}}{\cosh(\pi t)}\overline{\phi_{\Fc}(m,t)}\phi_{\Fc}(n,t)
\ll_{\epsilon} \frac{(m,s_{\star}s_{\square})^{\frac{1}{2}}(n,s_{\star}s_{\square})^{\frac{1}{2}}}{s_{\star}s_{\square}}(mns(1+|t|))^{\epsilon}.
\end{equation*}
\end{lem}
\begin{proof} This is part of \cite[Lemma 1]{BlomerKloos}.
\end{proof}
Substituting this inequality into \eqref{eq:cont} yields the bound
$$\begin{aligned}
\CE^{s}(m,n;f) &\lesssim \frac{(m,s_{\star}s_{\square})^{\frac{1}{2}}(n,s_{\star}s_{\square})^{\frac{1}{2}}}{s_{\star}s_{\square}} \int_{-\infty}^{\infty} (1+|t|)^{\epsilon} |\widehat{f}(t)| dt \\
&\lesssim \int_{-\infty}^{\infty} (1+|t|)^{\epsilon} |\widehat{f}(t)| dt.
\end{aligned}$$
We split the integral up into three parts
$$\begin{aligned}
\CI_1 &= \pm[\tfrac{1}{12}||\alpha|^2-1|^{\frac{1}{2}}X,2||\alpha|^2-1|^{\frac{1}{2}}X],\\
\CI_2 &= [-\max\{1,X^{\frac{1}{2}}\},\max\{1,X^{\frac{1}{2}}\}] \backslash \CI_1, \\
\CI_3 &= \pm [\max\{1,X^{\frac{1}{2}}\},\infty) \backslash \CI_1.
\end{aligned}$$
For $\CI_1$ we use \eqref{eq:trivial} and arrive at
$$\begin{aligned}
\int_{\CI_1} (1+|t|)^{\epsilon} |\widehat{f}(t)| dt &\ll_{\epsilon} \int_{\CI_1} (1+|t|)^{\epsilon} \frac{1+|\log(X)|+\log^+(|\alpha|)}{||\alpha|^2-1|^{\frac{1}{2}}X} dt \\
& \ll_{\epsilon} (1+X)^{\epsilon}(1+|\alpha|)^{\epsilon} (1+|\log(X)|+\log^+(|\alpha|)) \\
&\lesssim 1.
\end{aligned}$$
For $\CI_2$ we use \eqref{eq:trivial} again and arrive at
$$\begin{aligned}
\int_{\CI_2} (1+|t|)^{\epsilon} |\widehat{f}(t)| dt &\ll_{\epsilon} \int_{\CI_2} (1+|t|)^{\epsilon} \frac{1+|\log(X)|+\log^+(|\alpha|)}{1+X^{\frac{1}{2}}} dt \\
& \ll_{\epsilon} (1+X)^{\epsilon} (1+|\log(X)|+\log^+(|\alpha|)) \\
& \lesssim 1.
\end{aligned}$$
For $\CI_3$ we use \eqref{eq:nonholo4} and arrive at
$$\begin{aligned}
\int_{\CI_3} (1+|t|)^{\epsilon} |\widehat{f}(t)| dt &\ll_{\epsilon} \int_{\CI_3} |t|^{-\frac{3}{2}+\epsilon}\left(1+\left( \frac{X}{|t|} \right)^{\frac{1}{2}} \right) dt \\
&\ll_{\epsilon} \min\{1,X^{-\frac{1}{4}+\epsilon}\}+X^{\frac{1}{2}}\min\{1,X^{-\frac{1}{2}+\epsilon}\} \\
& \lesssim 1.
\end{aligned}$$
This concludes the proof of \eqref{eq:cont}.

\subsection{The Holomorphic Spectrum} The goal of this section is to prove the following inequality

\begin{equation}
\CH^{s}(m,n;f) \lesssim 1+X. 
\label{eq:holofull}
\end{equation}
In order to prove this inequality we choose our orthonormal basis as in \eqref{eq:holONB}. Then 
$$\begin{aligned}
\CH^s(m,n;f) = &  \frac{1}{\pi}\sum_{\substack{k \equiv 0\mod(2)\\k>0}} \sum_{r|s} \sum_{\substack{h \in S_k(r)\\ \text{new}}} \sum_{d | \frac{s}{r}} \frac{i^k \Gamma(k)}{(4 \pi \sqrt{mn})^{k-1}} \overline{\psi_{h^d}(m)}\psi_{h^d}(n) \widetilde{f}(k-1) \\
\lesssim & \sum_{\substack{k \equiv 0 \mod(2)\\k>0}} \sum_{r|s} \sum_{\substack{h \in S_k(r)\\ \text{new}}} \sum_{d | \frac{s}{r}} \frac{\Gamma(k)}{(4 \pi )^{k-1}} \frac{s}{r} |\psi_{h}(1)|^2 |\widetilde{f}(k-1)| \\
\lesssim &  \sum_{\substack{k \equiv 0 \mod(2)\\k>0}} \sum_{r|s} \sum_{\substack{h \in S_k(r)\\ \text{new}}} \frac{1}{r} |\widetilde{f}(k-1)| \\
\lesssim & \sum_{\substack{k \equiv 0 \mod(2)\\k> 0}} k^{1+\epsilon}|\widetilde{f}(k-1)|,
\end{aligned}$$
where we have made use of \eqref{eq:fourierbound}, \eqref{eq:firstcoef}, and $\dim S_k(r) \ll rk$ . The latter sum we split up into $k \le 9$ and $k > 9$. Using \eqref{eq:trivial} we find
$$
\sum_{\substack{k \equiv 0 \mod(2)\\9 \ge k>0}} k^{1+\epsilon}|\widetilde{f}(k-1)| \ll 1+ |\log(X)|+\log^+(|\alpha|) \lesssim 1.
$$
We also find
$$
\sum_{\substack{k \equiv 0 \mod(2)\\ k>9}} k^{1+\epsilon}|\widetilde{f}(k-1)| \le \CS_1+\CS_2+\CS_3+\CS_4,
$$
where
$$\begin{aligned}
\CS_1 & = \sum_{\substack{k \equiv 0 \mod(2)\\ k>9}} k^{1+\epsilon}\left|\int_{0}^{\frac{k-1}{2}}J_{k-1}(y)f(y)\frac{dy}{y}\right|,\\
\CS_2 & = \sum_{\substack{k \equiv 0 \mod(2)\\ k>9}} k^{1+\epsilon}\left|\int_{\frac{k-1}{2}}^{(k-1)-(k-1)^{\frac{1}{3}}}J_{k-1}(y)f(y)\frac{dy}{y}\right|,\\
\CS_3 & = \sum_{\substack{k \equiv 0 \mod(2)\\ k>9}} k^{1+\epsilon}\left|\int_{(k-1)-(k-1)^{\frac{1}{3}}}^{(k-1)+(k-1)^{\frac{1}{3}}}J_{k-1}(y)f(y)\frac{dy}{y}\right|,\\
\CS_4 & = \sum_{\substack{k \equiv 0 \mod(2)\\ k>9}} k^{1+\epsilon}\left|\int_{(k-1)+(k-1)^{\frac{1}{3}}}^{\infty}J_{k-1}(y)f(y)\frac{dy}{y}\right|.\\
\end{aligned}$$
Using \eqref{eq:holosplit1} we find
$$
\CS_1 \ll_{\epsilon} \sum_{k > 9} k^{\frac{1}{2}+\epsilon}e^{-\frac{2}{5}k} \ll_{\epsilon} 1.
$$
Using \eqref{eq:holosplit2} we find
$$
\CS_2 \ll_{\epsilon} \sum_{X/3 \le k-1\le 4X} k^{\epsilon} \lesssim 1+X.
$$
Using \eqref{eq:holosplit3} we find
$$
\CS_3 \ll_{\epsilon} \sum_{3X/16 \le k-1 \le 3X} k^{\epsilon} \lesssim 1+X.
$$
Using \eqref{eq:holosplit4} we find
$$
\CS_4 \ll_{\epsilon} \sum_{ 3X/2 \ge k-1 > 8} k^{\epsilon} \left( \frac{X}{k} \right)^{\frac{1}{2}} \lesssim 1+X.
$$
The claim \eqref{eq:holofull} now follows.\\

\subsection{The Non-Holomorphic Spectrum}

In this section we shall prove the following two estimates
\begin{multline}
\CM^s(m,n; f) + 2 \pi \sum_{t_h \in i[0,\theta]} \frac{\sqrt{mn} \cdot \overline{\rho_h(m)}\rho_h(n)}{\cos(\pi |t_h|)} \int_{\frac{X}{2}}^X Y_{2|t_h|}(x)e^{i\alpha x}\frac{dx}{x}  \\
\lesssim  \left(\frac{C}{T}\right)^{\frac{1}{2}}+(1+|\alpha|)X+\left(1+\frac{T}{C}X^{-2\theta}\right) \Biggl(1+ \frac{m^{\frac{1}{4}}(m,s)^{\frac{1}{4}}+n^{\frac{1}{4}}(n,s)^{\frac{1}{4}}}{s^{\frac{1}{2}}} \\
+ \min \left\{\frac{(mn)^{\frac{\theta}{2}+\frac{1}{8}}(mn,s)^{\frac{1}{8}}}{s^{\frac{1}{2}}}, \frac{(mn)^{\frac{1}{4}}(mn,s)^{\frac{1}{4}}}{s} \right\} \Biggr) 
\label{eq:Maassfull}
\end{multline}
and for $|\alpha|<1$ also
\begin{equation}\begin{aligned}
\CM^s(&m,n; f) + 2 \pi\sum_{t_h \in i[0,\theta]} \frac{\sqrt{mn} \cdot \overline{\rho_h(m)}\rho_h(n)}{\cos(\pi |t_h|)} \int_{\frac{X}{2}}^X Y_{2|t_h|}(x)e^{i\alpha x}\frac{dx}{x}  \\
\lesssim & (1-|\alpha|)^{-\frac{1}{2}-\epsilon}  \Biggl[ \left(\frac{C}{T}\right)^{\frac{1}{2}}+\left(1+\frac{T}{C}X^{-2\theta}\right) \\
& \quad \times \Biggl(1+ \frac{m^{\frac{1}{8}}(m,s)^{\frac{1}{8}}+n^{\frac{1}{8}}(n,s)^{\frac{1}{8}}}{s^{\frac{1}{4}}} \min\left\{ (mn)^{\frac{\theta}{2}} , \frac{m^{\frac{1}{8}}(m,s)^{\frac{1}{8}}+n^{\frac{1}{8}}(n,s)^{\frac{1}{8}}}{s^{\frac{1}{4}}} \right \} \\
& \quad \quad \quad \quad \quad \quad\quad \quad\quad \quad + \min \left\{\frac{(mn)^{\frac{3\theta}{4}+\frac{1}{16}}(mn,s)^{\frac{1}{16}}}{s^{\frac{1}{4}}}, \frac{(mn)^{\frac{1}{4}}(mn,s)^{\frac{1}{4}}}{s} \right\} \Biggr) \Biggr].
\label{eq:Maassfullalphadependent}
\end{aligned}\end{equation}
We shall require the following proposition.
\begin{prop} Let $A \ge 1$ and $n \in \BN$. Then we have for the group $\Gamma_0(s)$
$$
\sum_{\substack{|t_h| \le A}} \frac{n}{\cosh(\pi t_h)} |\rho_h(n)|^2 \ll_{\epsilon} A^2+\frac{\sqrt{n}}{s}(n,s)^{\frac{1}{2}}(ns)^{\epsilon}.
$$
\label{prop:Maassavg}
\end{prop}
\begin{proof} For the full modular group this is due to Kuznetov \cite[Eq. (5.19)]{Kuz1} and only minor modifications yields the above, see for example \cite[Lemma 2.9]{BerkeShiftedConv}  or \cite[Theorem 9]{APKloosterman}.
\end{proof}

Let us first prove \eqref{eq:Maassfull}. We split the summation over $t_h$ in $\CM^s(m,n;f)$ into various ranges $\CI_1,\dots,\CI_4$ which are treated individually. They are
$$\begin{aligned}
\CI_1 &= \left[0,\max\left\{1,X^{\frac{1}{2}}\right\}\right] ,\\
\CI_2 &= \left[\tfrac{1}{12}||\alpha|^2-1|^{\frac{1}{2}}X,2||\alpha|^2-1|^{\frac{1}{2}}X\right]\backslash \CI_1, \\
\CI_3 &= \left[\max\left\{1,X^{\frac{1}{2}}\right\}, \infty \right) \backslash \CI_2, \\
\CI_4 &= i\left[ 0, \frac{1}{2} \right].
\end{aligned}$$

The first way to treat the range $\CI_1$ is to choose the basis \eqref{eq:MaassONB} and use \eqref{eq:Maassfourierbound} as well as \eqref{eq:Maassnorm}:
\begin{equation*}\begin{aligned}
\sum_{t_h \in \CI_1} \frac{\sqrt{mn}}{\cosh(\pi t_h)} \overline{\rho_h(m)}\rho_h(n) \widehat{f}(t_h) &\lesssim (mn)^{\theta} \sum_{r|s} \frac{1}{r} \sum_{\substack{t_h \in \CI_1 \\ \text{new of level }r}} \sum_{d|\frac{s}{r}} (1+|t_h|)^{\epsilon} \sup_{t \in \CI_1} |\widehat{f}(t)|. \\
\end{aligned}\end{equation*}
Next we use \eqref{eq:trivial} to bound the transform and a uniform Weyl law to bound the number of Maass forms $h$ of level $r$ with $t_h \le T$ by $r^{1+\epsilon}T^2$ (see for example \cite[Corollary 3.2.3.]{uniformweyl}). We arrive at the bound
\begin{equation}
\lesssim (mn)^{\theta} \left(1+X^{\frac{1}{2}}\right).
\label{eq:CI1a}
\end{equation} 
A second way to treat the range $\CI_1$ is to apply the Cauchy-Schwarz inequality in conjunction with Proposition \ref{prop:Maassavg} and \eqref{eq:trivial}:
\begin{equation}\begin{aligned}
&\sum_{t_h \in \CI_1} \frac{\sqrt{mn}}{\cosh(\pi t_h)} \overline{\rho_h(m)}\rho_h(n) \widehat{f}(t_h) \\
\le & \left( \sum_{t_h \in \CI_1} \frac{m}{\cosh(\pi t_h)} |\rho_h(m)|^2 \right)^{\frac{1}{2}} \left( \sum_{t_h \in \CI_1} \frac{n}{\cosh(\pi t_h)} |\rho_j(n)|^2 \right)^{\frac{1}{2}} \sup_{t \in \CI_1} |\widehat{f}(t)| \\
\lesssim & \left( 1+X+\frac{\sqrt{m}}{s}(m,s)^{\frac{1}{2}} \right)^{\frac{1}{2}} \left( 1+X+\frac{\sqrt{n}}{s}(n,s)^{\frac{1}{2}} \right)^{\frac{1}{2}} \frac{1}{1+X^{\frac{1}{2}}} \\
\lesssim & 1+X^{\frac{1}{2}}+\frac{m^{\frac{1}{4}}(m,s)^{\frac{1}{4}}+n^{\frac{1}{4}}(n,s)^{\frac{1}{4}}}{s^{\frac{1}{2}}}+\frac{(mn)^{\frac{1}{4}}(mn,s)^{\frac{1}{4}}}{s(1+X^{\frac{1}{2}})}.
\label{eq:CI1b}
\end{aligned}\end{equation}

The range $\CI_2$ we treat in exactly the same manner and we arrive at the inequalities
\begin{equation}\begin{aligned}
\sum_{t_h \in \CI_2} \frac{\sqrt{mn}}{\cosh(\pi t_h)} \overline{\rho_h(m)}\rho_h(n) \widehat{f}(t_h) &\lesssim (mn)^{\theta} \frac{\left(1+||\alpha|^2-1|^{\frac{1}{2}}X\right)^2}{1+||\alpha|^2-1|^{\frac{1}{2}}X} \\
&\lesssim (mn)^{\theta}\left(1+||\alpha|^2-1|^{\frac{1}{2}}X\right)
\label{eq:CI2a}
\end{aligned}\end{equation}
and
\begin{equation}\begin{aligned}
 &\sum_{t_h \in \CI_2} \frac{\sqrt{mn}}{\cosh(\pi t_h)} \overline{\rho_h(m)}\rho_h(n) \widehat{f}(t_h) \\
\lesssim & \frac{\left(1+||\alpha|^2-1|^{\frac{1}{2}}X+\frac{m^{\frac{1}{4}}(m,s)^{\frac{1}{4}}}{s^{\frac{1}{2}}}\right)\left(1+||\alpha|^2-1|^{\frac{1}{2}}X+\frac{n^{\frac{1}{4}}(n,s)^{\frac{1}{4}}}{s^{\frac{1}{2}}}\right)}{1+X^{\frac{1}{2}}+||\alpha|^2-1|^{\frac{1}{2}}X} \\
\lesssim & 1+||\alpha|^2-1|^{\frac{1}{2}}X + \frac{m^{\frac{1}{4}}(m,s)^{\frac{1}{4}}+n^{\frac{1}{4}}(n,s)^{\frac{1}{4}}}{s^{\frac{1}{2}}} + \frac{(mn)^{\frac{1}{4}}(mn,s)^{\frac{1}{4}}}{s(1+||\alpha|^2-1|^{\frac{1}{2}}X)}.
\label{eq:CI2b}
\end{aligned}\end{equation}

The range $\CI_3$ we further split into dyadic ranges
$$
\CI_3(l)= [2^l\max\{1,X^{\frac{1}{2}}\},2^{l+1}\max\{1,X^{\frac{1}{2}}\}]\backslash \CI_2, \quad l \ge 0.
$$
Again we can estimate
\begin{equation}
\sum_{t_h \in \CI_3(l)} \frac{\sqrt{mn}}{\cosh(\pi t_h)} |\rho_h(m)\rho_h(n)| \lesssim (mn)^{\theta} 2^{2l} (1+X)
\label{eq:helpa}
\end{equation}
and
\begin{multline}
\sum_{t_h \in \CI_3(l)} \frac{\sqrt{mn}}{\cosh(\pi t_h)} |\rho_h(m)\rho_h(n)| \\
\lesssim 2^{2l} (1+X) + 2^l\left(1+X^{\frac{1}{2}}\right) \frac{m^{\frac{1}{4}}(m,s)^{\frac{1}{4}}+n^{\frac{1}{4}}(n,s)^{\frac{1}{4}}}{s^{\frac{1}{2}}}+\frac{(mn)^{\frac{1}{4}}(mn,s)^{\frac{1}{4}}}{s} ,
\label{eq:helpb}
\end{multline}
However this time we use \eqref{eq:nonholo4} and \eqref{eq:nonholo5} to deal with the transform. We have
\begin{equation}
\sup_{t \in \CI_3(l)} |\widehat{f}(t)| \lesssim \begin{cases}  \min \left\{ \frac{1+X^{\frac{1}{2}}}{2^{2l}(1+X)}, \frac{C}{T} \frac{1+X^{\frac{3}{2}}}{2^{4l}(1+X)^2} \right\} , & \text{for } l \le \log_2(\max\{1,X^{\frac{1}{2}}\}) , \\ \min\left\{ \frac{1}{2^{\frac{3}{2}l}(1+X)^{\frac{3}{4}}} , \frac{C}{T} \frac{1}{2^{\frac{5}{2}l}(1+X)^{\frac{5}{4}}} \right\} ,& \text{for } l > \log_2(\max\{1,X^{\frac{1}{2}}\}) . \end{cases}
\label{eq:helpc}
\end{equation}
Combining \eqref{eq:helpa}, \eqref{eq:helpb} and \eqref{eq:helpc} we find that the contribution stemming from $l \le \log_2(\max\{1,X^{\frac{1}{2}}\})$ is
\begin{equation}\begin{aligned}
&\lesssim \sum_{l \le \log_2(\max\{1,X^{\frac{1}{2}}\})} \Biggl(1+ X^{\frac{1}{2}} + 2^{-l} \frac{m^{\frac{1}{4}}(m,s)^{\frac{1}{4}}+n^{\frac{1}{4}}(n,s)^{\frac{1}{4}}}{s^{\frac{1}{2}}} \\
& \quad \quad \quad \quad \quad \quad \quad \quad \quad \quad \quad \quad  \quad \quad \quad  +\min \left\{ (mn)^{\theta}(1+X)^{\frac{1}{2}}, 2^{-2l} \frac{(mn)^{\frac{1}{4}}(mn,s)^{\frac{1}{4}}}{s(1+X)^{\frac{1}{2}}}\right \} \Biggr) \\
&\lesssim 1+ X^{\frac{1}{2}} + \frac{m^{\frac{1}{4}}(m,s)^{\frac{1}{4}}+n^{\frac{1}{4}}(n,s)^{\frac{1}{4}}}{s^{\frac{1}{2}}} + \min \left\{ \frac{(mn)^{\frac{\theta}{2}+\frac{1}{8}}(mn,s)^{\frac{1}{8}}}{s^{\frac{1}{2}}} , \frac{(mn)^{\frac{1}{4}}(mn,s)^{\frac{1}{4}}}{s}\right \},
\label{eq:CI3a}
\end{aligned}\end{equation}
the contribution from $l > \log_2(\max\{1,X^{\frac{1}{2}}\})$ is
\begin{equation}\begin{aligned}
&\lesssim \sum_{l > \log_2(\max\{1,X^{\frac{1}{2}}\})} \Biggl( \left( \frac{C}{T} \right)^{\frac{1}{2}+\delta} 2^{-\delta l} +  2^{-\frac{l}{2}} \frac{m^{\frac{1}{4}}(m,s)^{\frac{1}{4}}+n^{\frac{1}{4}}(n,s)^{\frac{1}{4}}}{s^{\frac{1}{2}}} \\
& \quad \quad \quad \quad\quad \quad\quad \quad\quad \quad\quad \quad\quad \ \,+ 2^{-\frac{l}{2}} \min \Biggl \{ \frac{(mn)^{\frac{\theta}{2}+\frac{1}{8}}(mn,s)^{\frac{1}{8}}}{s^{\frac{1}{2}}}, \frac{(mn)^{\frac{1}{4}}(mn,s)^{\frac{1}{4}}}{s} \Biggr\} \Biggr) \\
&\lesssim \left( \frac{C}{T} \right)^{\frac{1}{2}} + \frac{m^{\frac{1}{4}}(m,s)^{\frac{1}{4}}+n^{\frac{1}{4}}(n,s)^{\frac{1}{4}}}{s^{\frac{1}{2}}} + \min \left\{\frac{(mn)^{\frac{\theta}{2}+\frac{1}{8}}(mn,s)^{\frac{1}{8}}}{s^{\frac{1}{2}}}, \frac{(mn)^{\frac{1}{4}}(mn,s)^{\frac{1}{4}}}{s} \right\}  ,
\label{eq:CI3b}
\end{aligned}\end{equation}
for a sufficiently small $\delta>0$.

For the contribution from $\CI_4$ we first note that we have $|t_h|\le \theta$ for $t_h \in \CI_4$ by \cite{KimSar}. We first insert \eqref{eq:exasymp} and further find
\begin{multline}
4 \pi \sum_{t_h \in i[0,\theta]} \frac{\sqrt{mn}}{\cosh(\pi t_h)} \overline{\rho_h(m)}\rho_h(n) \left( -\frac{1}{2}\int_{\frac{X}{2}}^{X}Y_{2|t_h|}(x)e^{i \alpha x} \frac{dx}{x} + O_{\epsilon}\left( 1+ \frac{T}{C} X^{-2|t_h|-\epsilon} \right) \right) \\
= - 2 \pi \sum_{t_h \in i[0,\theta]} \frac{\sqrt{mn} \cdot \overline{\rho_h(m)}\rho_h(n)}{\cos(\pi |t_h|)} \int_{\frac{X}{2}}^X Y_{2|t_h|}(x)e^{i\alpha x}\frac{dx}{x}  \\
+ O_{\epsilon}\left( \left(1+ \frac{T}{C}X^{-2\theta-\epsilon} \right) \min\left\{ (mn)^{\theta}, 1+\frac{m^{\frac{1}{4}}(m,s)^{\frac{1}{4}}+n^{\frac{1}{4}}(n,s)^{\frac{1}{4}}}{s^{\frac{1}{2}}}+\frac{(mn)^{\frac{1}{4}}(mn,s)^{\frac{1}{4}}}{s}  \right\} \right).
\label{eq:CI4a}
\end{multline}
Combining the minimum of \eqref{eq:CI1a} and \eqref{eq:CI1b}, the minimum of \eqref{eq:CI2a} and \eqref{eq:CI2b}, \eqref{eq:CI3a}, \eqref{eq:CI3b} with \eqref{eq:CI4a} gives \eqref{eq:Maassfull}.\\

Let us now turn our attention to \eqref{eq:Maassfullalphadependent}. This time we split up into the intervals
$$\begin{aligned}
\CI_1 &= \left[0,1\right] ,\\
\CI_2 &= \left[1, \infty \right), \\
\CI_3 &= i\left[ 0, \frac{1}{2} \right].
\end{aligned}$$
By making use of \eqref{eq:trivial} we find that the contribution from $\CI_1$ is bounded by
\begin{equation}
\lesssim \min \left \{ (mn)^{\theta}, 1+\frac{m^{\frac{1}{4}}(m,s)^{\frac{1}{4}}+n^{\frac{1}{4}}(n,s)^{\frac{1}{4}}}{s^{\frac{1}{2}}}+\frac{(mn)^{\frac{1}{4}}(mn,s)^{\frac{1}{4}}}{s} \right \}.
\label{eq:2CI1}
\end{equation}
As before we split up $\CI_2$ into dyadic ranges $\CI_2(l)=[2^l,2^{l+1}]$, $l \ge 0$ and use
\begin{equation*}
\sup_{t \in \CI_2(l)} |\widehat{f}(t)| \lesssim \min \left\{ (1-|\alpha|)^{-\frac{1}{4}}2^{-\frac{3}{2}l}, \frac{C}{T} (1-|\alpha|)^{-\frac{3}{4}}  2^{-\frac{5}{2}l} \right \},
\end{equation*}
which follows from \eqref{eq:nonholo4} and \eqref{eq:nonholo5}. Thus we find that the contribution from $\CI_2$ is bounded by
\begin{equation}\begin{aligned}
 \lesssim & (1-|\alpha|)^{-\frac{1}{2}-\frac{\delta}{2}} \sum_{l \ge 0} \Biggl( \left( \frac{C}{T} \right)^{\frac{1}{2}+\delta} 2^{-\delta l} \\
&+ \min \left\{(mn)^{\frac{\theta}{2}-\theta\delta}\frac{m^{\frac{1}{8}+\frac{\delta}{4}}(m,s)^{\frac{1}{8}+\frac{\delta}{4}}+n^{\frac{1}{8}+\frac{\delta}{4}}(n,s)^{\frac{1}{8}+\frac{\delta}{4}}}{s^{\frac{1}{4}+\frac{\delta}{2}}}2^{-\delta l},\frac{m^{\frac{1}{4}}(m,s)^{\frac{1}{4}}+n^{\frac{1}{4}}(n,s)^{\frac{1}{4}}}{s^{\frac{1}{2}}}2^{-\frac{1}{2}l} \right\} \\ & + \min \left\{(mn)^{\frac{3\theta}{4}-\theta\delta}\frac{(mn)^{\frac{1}{16}+\frac{\delta}{4}}(mn,s)^{\frac{1}{16}+\frac{\delta}{4}}}{s^{\frac{1}{4}+\delta}}2^{-2\delta l},\frac{(mn)^{\frac{1}{4}}(mn,s)^{\frac{1}{4}}}{s}2^{-\frac{3}{2}l} \right\} \Biggr) \\
 \lesssim & (1-|\alpha|)^{-\frac{1}{2}-\epsilon} \Biggl( \left( \frac{C}{T} \right)^{\frac{1}{2}} + \frac{m^{\frac{1}{8}}(m,s)^{\frac{1}{8}}+n^{\frac{1}{8}}(n,s)^{\frac{1}{8}}}{s^{\frac{1}{4}}} \min \left\{(mn)^{\frac{\theta}{2}},\frac{m^{\frac{1}{8}}(m,s)^{\frac{1}{8}}+n^{\frac{1}{8}}(n,s)^{\frac{1}{8}}}{s^{\frac{1}{4}}} \right\} \\ & + \min \left\{\frac{(mn)^{\frac{3\theta}{4}+\frac{1}{16}}(mn,s)^{\frac{1}{16}}}{s^{\frac{1}{4}}},\frac{(mn)^{\frac{1}{4}}(mn,s)^{\frac{1}{4}}}{s} \right\} \Biggr)
\label{eq:2CI2}
\end{aligned}\end{equation}
for $\delta>0$ small enough. The contribution from $\CI_3$ is the same as in \eqref{eq:CI4a}. Combining \eqref{eq:2CI1}, \eqref{eq:2CI2}, and \eqref{eq:CI4a} gives \eqref{eq:Maassfullalphadependent}.

\subsection{Putting things together}
In order to show Theorem \ref{thm:dyadic} we add up all the inequalities \eqref{eq:sharperror}, \eqref{eq:cont}, \eqref{eq:holofull}, \eqref{eq:Maassfull} respectively \eqref{eq:Maassfullalphadependent}, and make the choice $T=O(s^{\frac{2}{3}}C^{\frac{2}{3}})$, which is allowed since $s \ll \min \{(mn)^{\frac{1}{4}}, C^{\frac{1}{2}}\}$. One may note that we have$$
\frac{(m,n)^{\frac{1}{2}}}{\sqrt{C}} \le X^{\frac{1}{2}} \le 1+X
$$
and 
$$
\frac{T}{C}X^{-2\theta} \ll s^{\frac{2}{3}-4 \theta} C^{2\theta - \frac{1}{3}} \cdot s^{4\theta} (mn)^{-\theta} \ll 1.
$$
Theorem \ref{thm:1} follows now at once by estimating the range $c \le (1+|\alpha|^{\frac{2}{3}}) s^{\frac{2}{3}}(mn)^{\frac{1}{3}}$ trivially using the Weil bound, which gives
$$
\sum_{\substack{c \le (1+|\alpha|^{\frac{2}{3}})s^{\frac{2}{3}}(mn)^{\frac{1}{3}} \\ c \equiv 0 \mod(s)}} \frac{1}{c} |S(m,n;c)| \lesssim \frac{\left((1+|\alpha|^{\frac{2}{3}})s^{\frac{2}{3}}(mn)^{\frac{1}{3}}\right)^{\frac{1}{2}}}{s} \lesssim (1+|\alpha|^{\frac{1}{3}}) \frac{(mn)^{\frac{1}{6}}}{s^{\frac{2}{3}}}.
$$
For the remaining range $(1+|\alpha|^{\frac{2}{3}})s^{\frac{2}{3}}(mn)^{\frac{1}{6}} \le c \le C$ we use Theorem \ref{thm:dyadic}. Furthermore note that
$$
\int_1^{\infty} |Y_{2t}(x)| \frac{dx}{x} \ll \int_1^{\infty} x^{-\frac{3}{2}}dx \ll 1
$$
uniformly for $t \le \theta$ and hence we have
\begin{multline*}
\sum_{t_h \in i[0,\theta]} \frac{\sqrt{mn} \cdot \overline{\rho_h(m)}\rho_h(n)}{\cos(\pi |t_h|)} \int_{1}^{\infty} Y_{2|t_h|}(x)e^{i\alpha x}\frac{dx}{x} \\
 \lesssim \min \left \{ (mn)^{\theta}, 1+\frac{m^{\frac{1}{4}}(m,s)^{\frac{1}{4}}+n^{\frac{1}{4}}(n,s)^{\frac{1}{4}}}{s^{\frac{1}{2}}}+\frac{(mn)^{\frac{1}{4}}(mn,s)^{\frac{1}{4}}}{s}  \right \}.
\end{multline*}
This proves Theorem \ref{thm:1}. In order to prove Corollary \ref{cor:2} we need to show
\begin{equation*}
\sum_{t_h \in i[0,\theta]} \frac{\sqrt{mn} \cdot \overline{\rho_h(m)}\rho_h(n)}{\cos(\pi |t_h|)} \int_{X}^{1} Y_{2|t_h|}(x)e^{i\alpha x}\frac{dx}{x} \lesssim C^{2\theta},
\end{equation*}
when $C \ge \sqrt{mn}$. This follows from the two estimates
$$
\int_{X}^{1} |Y_{2t}(x)| \frac{dx}{x} \ll_{\epsilon} \int_{X}^{1} x^{-2\theta -1- \epsilon}dx \ll_{\epsilon} X^{-2\theta - \epsilon}
$$
and
$$
\sum_{t_h \in i[0,\theta]} \frac{\sqrt{mn} \cdot |\rho_h(m)\rho_h(n)|}{\cos(\pi |t_h|)} \ll_{\epsilon} (mn)^{\theta+\epsilon}.
$$
Theorem \ref{thm:2} is proved analogously.

\section{Transform estimates}
\label{sec:est}
In this section we prove the claimed upper bounds in Lemma \ref{lem:upperbnds} on the transforms of $f$. Since all the estimates are very different in nature we split them up into multiple lemmata. We generally follow the arguments of \cite{SarTsim} and \cite{generalkuznetsov}, but tweak them to account for our introduced twist. First we shall need two preliminary lemmata, which will be used frequently.

\begin{lem} \label{lem:intbound} Let $F,G \in C([A,B],\BC)$ with $G$ having a continuous derivative. Then
$$
\left|\int_A^B F(x)G(x) dx \right| \ll \left( \| G \|_{\infty}+\|G' \|_{1} \right) \sup_{C \in [A,B]} \left| \int_A^C F(x)dx \right |.
$$
\end{lem}
\begin{proof} We integrate by parts and find
$$
\int_A^B F(x)G(x) dx = \int_A^B F(x) dx \cdot G(B) - \int_A^B \int_A^y F(x)dx \cdot G'(y) dy,
$$
from which the first statement is trivially deduced.
\end{proof}

\begin{lem}
Let $G,H \in C^1([A,B],\BC)$ and assume $G$ has a zero and $H'$ has at most $K$ zeros. Then we have
$$
\| GH \|_{\infty}+\|(GH)'\|_{1} \ll_K \|G'\|_{1} \|H\|_{\infty}.
$$
\label{lem:intbound2}
\end{lem}
\begin{proof} We have $\| GH \|_{\infty} \le \| G \|_{\infty}\|H \|_{\infty}$ and $\|G\|_{\infty} \le \|G'\|_{1}$ since we have $G(b)= \int_a^b G'(x) dx$, where $a$ is a zero of $G$. Furthermore we have
$$
\|(GH)'\|_{1} \le \|G'H\|_1+\|GH'\|_1 \le \|G'\|_1 \|H\|_{\infty}+\|G\|_{\infty}\|H'\|_1 \le \|G'\|_1(\|H\|_{\infty}+\|H'\|_1)
$$
and
$$
\|H'\|_1 \le 2 (K+1) \|H\|_{\infty}
$$
by splitting up the integral into intervals on which $H'$ has a constant sign.
\end{proof}

\begin{lem} Let $f$ be as in the beginning of Section \ref{sec:proof} and $|\alpha|\le 1$ then we have
\begin{align*}
\widetilde{f}(t) \ll \frac{1+|\log(X)|}{1+X^{\frac{1}{2}}+|1-|\alpha|^2|^{\frac{1}{2}}X}, \quad \forall t \in \BR.
\end{align*}
\end{lem}

\begin{proof} We follow the proof of Lemma 7.1 in \cite{generalkuznetsov} and Proposition 5 in \cite{SarTsim}. To prove the first statement we use the Bessel representation
$$
J_t(x)= \frac{1}{2\pi} \int_0^{2\pi} e^{i(x\sin \xi-t\xi)} d\xi.
$$
Integration by parts yields
$$\begin{aligned}
\int_0^{\infty}e^{i x \sin \xi}\frac{f(x)}{x}dx&=\int_0^{\infty}e^{i x (\sin \xi+\alpha)}\frac{g(x)}{x}dx \\
&= \frac{i}{\sin\xi+\alpha} \int_0^{\infty} e^{i x (\sin\xi+\alpha)} \left( \frac{g(x)}{x} \right)' dx \\
&\ll \min \left\{  1, X^{-1}|\sin\xi+\alpha|^{-1}  \right \}.
\end{aligned}$$
Thus we find
$$\begin{aligned}
\widetilde{f}(t) & \ll \int_0^{2\pi} \min \left\{ 1, X^{-1}|\sin\xi+\alpha|^{-1} \right\}d\xi 
\end{aligned}$$
Now clearly $\widetilde{f}(t) \ll 1$. For $X \ge 1 $ we can do better though. We have $|\sin \xi+ \alpha| \ge ||\sin \xi|-|\alpha||$, thus we may assume $\xi \in [0,\frac{\pi}{2}]$ and $\alpha \ge 0$. Set $\alpha = \sin \phi$ with $\phi \in [0, \frac{\pi}{2}]$. Then we have
$$
\sin \xi - \alpha = 2 \sin \left( \frac{\xi-\phi}{2} \right) \sin\left( \frac{\pi-\xi-\phi}{2} \right).
$$
Now for $x \in [-\frac{\pi}{2},\frac{\pi}{2}]$ we have $|\sin(x)| \asymp |x|$, thus
$$\begin{aligned}
\widetilde{f}(t) & \ll \int_0^{\frac{\pi}{2}} \min \left\{ 1, X^{-1}|\xi-\phi|^{-1}|\pi-\xi-\phi|^{-1} \right\}d\xi \\ 
& \ll \int_0^{\frac{\pi}{2}} \min \left\{ 1, X^{-1}|\xi-\phi|^{-1}|\tfrac{\pi}{2}-\phi|^{-1},X^{-1}|\xi-\phi|^{-2} \right\}d\xi \\
& \ll \min \left\{  \frac{1+\log(X)}{|\tfrac{\pi}{2}-\phi|X}, X^{-\frac{1}{2}}    \right \}.
\end{aligned}$$
Now we just have to note that $\frac{\pi}{2}-\phi  \asymp \sin(\frac{\pi}{2}-\phi)=\sqrt{1-|\alpha|^2}$.
\end{proof}

\begin{lem} Let $f$ be as in the beginning of Section \ref{sec:proof} and $|\alpha|\ge 1$ then we have
\begin{align*}
\widetilde{f}(t) \ll \frac{1+|\log(X)|}{1+X^{\frac{1}{2}}+||\alpha|^2-1|^{\frac{1}{2}}X}, \quad \forall t \in \BR.
\end{align*}
\end{lem}

\begin{proof} As before we find $\widetilde{f}(t) \ll 1$ and for $X \ge 1$ we have
$$\begin{aligned}
\widetilde{f}(t) & \ll \int_0^{\frac{\pi}{2}} \min \left\{ 1, X^{-1}(|\alpha|-|\sin \xi|)^{-1} \right\}d\xi \\
& \ll \int_0^{\frac{\pi}{2}} \min \left\{ 1, X^{-1}(|\alpha|-1+\tfrac{1}{\pi}(\tfrac{\pi}{2}-\xi)^2)^{-1} \right\}d\xi \\
& \ll \int_0^{\frac{\pi}{2}} \min \left\{ 1, X^{-1}(|\alpha|-1)^{-1}, X^{-1}(|\alpha|-1)^{-\frac{1}{2}}(\tfrac{\pi}{2}-\xi)^{-1} ,  X^{-1}(\tfrac{\pi}{2}-\xi)^2 \right\}d\xi\\
& \ll \min\left \{  \frac{1}{||\alpha|-1|X}, \frac{1+\log(X)}{||\alpha|-1|^{\frac{1}{2}}X} ,X^{-\frac{1}{2}}   \right \}.
\end{aligned}$$

\end{proof}

We also require some more refined estimates. For this we consider the different regions of the $J$-Bessel function.

\begin{lem} Let $f$ as in the beginning of Section \ref{sec:proof} and $|\alpha|\le 1$. Then we have for $t\ge 8$
\begin{align}
\int_0^{\frac{t}{2}} J_t(y) f(y) \frac{dy}{y} &\ll \one_{[2X/3, \infty)}(t) \cdot t^{-\frac{1}{2}} e^{-\frac{2}{5}t}, \nonumber \\
\int_{\frac{t}{2}}^{t-t^{\frac{1}{3}}} J_t(y)f(y) \frac{dy}{y} &\ll
\one_{[\frac{1}{4},\infty)}(X)\one_{[X/3,4X]}(t) \cdot t^{-1}(\log(t))^{\frac{2}{3}}, \nonumber \\
\int_{t-t^{\frac{1}{3}}}^{t+t^{\frac{1}{3}}} J_t(y) f(y) \frac{dy}{y} &\ll \one_{[\frac{1}{4},\infty)}(X)\one_{[3X/16,3X]}(t) \cdot t^{-1}, \nonumber \\ 
\int_{t+t^{\frac{1}{3}}}^{\infty} J_t(y) f(y) \frac{dy}{y} &\ll  \one_{[\frac{1}{4},\infty)}(X)\one_{[0,3X/2]}(t) \cdot t^{-1} \min \left \{ |1-|\alpha||^{-\frac{1}{4}}, \left(  \frac{X}{t} \right)^{\frac{1}{2}} \right \}, \label{eq:holosplitsmall4}
\end{align}
where $\one_{\CI}$ is the characteristic function of the interval $\CI$.
\end{lem}

\begin{proof}
We require some uniform estimates on the $J$-Bessel functions of real order.
For small argument we have exponential decay
\begin{equation}
0\le J_{t}(x) \le \frac{e^{-t F(0,x/t)}}{\left(1-(x/t)^2\right)^{\frac{1}{4}}\sqrt{2\pi t}}, \quad \forall x < t,
\label{eq:Jexpdecay}
\end{equation}
where $F(0,x)=\log \left( \frac{1+\sqrt{1-x^2}}{x} \right)-\sqrt{1-x^2}$. The left hand side follows from the fact that the first zero of the Bessel function of order $t$ is $>t$ and the right hand side follows from \cite[pp. 252-255]{ToBF}. We will also make use of Langer's formulas see \cite[pp. 30,89]{HTF}. The first formula is
\begin{equation}
J_{t}(x) = w^{-\frac{1}{2}}(w-\arctan(x))^{\frac{1}{2}} \left( \frac{\sqrt{3}}{2} J_{\frac{1}{3}}(z)-\frac{1}{2}Y_{\frac{1}{3}}(z) \right) + O(t^{-\frac{4}{3}}), \quad \forall x>t,
\label{eq:Langersmall}
\end{equation}
where
$$
w= \sqrt{\frac{x^2}{t^2}-1} \text{ and } z= t (w-\arctan(w)).
$$
The second one is
\begin{equation}
J_{t}(x) = \frac{1}{\pi} w^{-\frac{1}{2}}(\artanh(w)-w)^{\frac{1}{2}}K_{\frac{1}{3}}(z)+O(t^{-\frac{4}{3}}), \quad \forall x < t,
\label{eq:Langerlarge}
\end{equation}
where
$$
w=\sqrt{1-\frac{x^2}{t^2}} \text{ and } z= t(\artanh(w)-w).
$$
And finally for the transitional range $|x-t|\le t^{\frac{1}{3}}$ we have
\begin{equation}
J_{t}(x) \ll t^{-\frac{1}{3}},
\label{eq:Jtrans}
\end{equation}
by \cite[pp. 244-247]{ToBF}.

The first inequality follows directly from \eqref{eq:Jexpdecay}
$$
\int_0^{\frac{t}{2}} J_t(y) f(y) \frac{dy}{y} \ll t^{-\frac{1}{2}}e^{-\frac{2}{5}t} \cdot \frac{X}{X}.
$$
Note that if $X\le \frac{1}{2}$, then this covers everything, thus we may assume $X \ge \frac{1}{2}$ from now on. For the range $[\frac{t}{2},t-t^{\frac{1}{3}}]$ we use \eqref{eq:Langerlarge} and $z^{\frac{1}{2}}K_{\frac{1}{3}}(z) \ll e^{-z}, \ \forall z \ge 0$. Thus we find
$$
J_t(y) \ll (t^2-y^2)^{-\frac{1}{4}}e^{-z}+O(t^{-\frac{4}{3}}).
$$
Now if $y \le \min\{ t-9t^{\frac{1}{3}}(\log t)^{\frac{2}{3}}, t-t^{\frac{1}{3}}\}$ we have $z \ge \log t$ and thus $J_t(y) \ll t^{-\frac{4}{3}}$ otherwise we have $J_t(y) \ll t^{-\frac{1}{3}}$. We conclude
$$
\int_{\frac{t}{2}}^{t-t^{\frac{1}{3}}} J_t(y)f(y) \frac{dy}{y} \ll
t^{-\frac{4}{3}} \cdot \frac{X}{X}+t^{-\frac{1}{3}} \cdot \frac{t^{\frac{1}{3}}(\log(t))^{\frac{2}{3}}}{t}.
$$
For the range $t-t^{\frac{1}{3}} \le y \le t+t^{\frac{1}{3}}$ we use \eqref{eq:Jtrans} and get
$$
\int_{t-t^{\frac{1}{3}}}^{t+t^{\frac{1}{3}}} J_t(y) f(y) \frac{dy}{y} \ll t^{-\frac{1}{3}} \cdot \frac{t^{\frac{1}{3}}}{t}.
$$
We are left to deal with the range $t+t^{\frac{1}{3}} \le y$. We make a change of variable $y \to ty$ and we are left to estimate
\begin{equation}
\int_{1+t^{-\frac{2}{3}}}^{\infty} J_t(ty)e^{i\alpha ty} g(ty) \frac{dy}{y}.
\label{eq:left1}
\end{equation}
We make use of \eqref{eq:Langersmall} and find $z \gg 1$ in this range of $y$. By making use of Langer's formula \eqref{eq:Langersmall} we introduce an error of the size
$$
\ll t^{-\frac{4}{3}} \cdot \frac{X}{X},
$$
which is sufficient. Since $z \gg 1$ we are able to make use of the classical estimates
\begin{equation}\begin{aligned}
J_{\frac{1}{3}}(z) &= \sqrt{\frac{2}{\pi z}} \left( \cos(z-\tfrac{\pi}{6}-\tfrac{\pi}{4})+O(z^{-1}) \right), \\
Y_{\frac{1}{3}}(z) &= \sqrt{\frac{2}{\pi z}} \left( \sin(z-\tfrac{\pi}{6}-\tfrac{\pi}{4})+O(z^{-1}) \right).
\end{aligned}
\label{eq:classical}
\end{equation}
Inserting \eqref{eq:classical} into \eqref{eq:left1} introduces another error of the size
$$
t^{-\frac{1}{2}}\int_{1+t^{-\frac{2}{3}}}^{\infty} w^{-\frac{1}{2}} z^{-1} g(ty) \frac{dy}{y}, 
$$
where $w=\sqrt{y^2-1}$ and $z=t(w-\arctan(w))$. We have $z \gg t \min \{w^3,w\}$ and thus we are able to estimate the above as
$$\begin{aligned}
&\ll t^{-\frac{3}{2}} \int_{1+t^{-\frac{2}{3}}}^2 \frac{g(ty)}{(y^2-1)^{\frac{7}{4}}y}dy+t^{-\frac{3}{2}} \int_2^{\infty} \frac{g(ty)}{(y^2-1)^{\frac{3}{4}}y}dy \\
& \ll t^{-\frac{3}{2}} \int_{1+t^{-\frac{2}{3}}}^2 \frac{g(ty)y}{(y^2-1)^{\frac{7}{4}}}dy+t^{-\frac{3}{2}} \int_2^{\infty} \frac{g(ty)}{y^{\frac{5}{2}}}dy \\
& \ll \| g' \|_1 \cdot t^{-1} + t^{-\frac{3}{2}},
\end{aligned}$$
where we have made use of Lemmata \ref{lem:intbound} and \ref{lem:intbound2} with $F(y)=y(y^2-1)^{-\frac{7}{4}}$ and $G(y)=g(ty)$ respectively $F(y)=y^{-\frac{5}{2}}$ and $G(y)=g(ty)$. This is again sufficient. For the main term we have to consider
\begin{equation}
t^{-\frac{1}{2}} \int_{1+t^{-\frac{2}{3}}}^{\infty} e^{it(\pm\omega(y)+\alpha y)} \frac{g(ty)}{(y^2-1)^{\frac{1}{4}}y}dy,
\label{eq:holomain}
\end{equation}
where
$$\begin{aligned}
\omega(y) &=\sqrt{y^2-1}-\arctan{\sqrt{y^2-1}}, \\
\omega'(y) &= \frac{\sqrt{y^2-1}}{y}.
\end{aligned}$$
We would like to integrate $t(\pm\omega'(y)+\alpha)e^{it(\pm \omega(y)+\alpha y)}$ by parts, but for the sign `$-\sign(\alpha)$' and $y_0=(1-\alpha^2)^{-\frac{1}{2}}$ we have $\omega'(y_0)=|\alpha|$ and we pick up a stationary phase. Let us first assume $\alpha$ is close to $0$, such that $y_0<1+t^{-\frac{2}{3}}$. For $|\alpha| \ll t^{-\frac{1}{3}}$ or the sign `$\sign(\alpha)$' we have $|\pm\omega'(1+t^{-\frac{2}{3}})+\alpha| \gg t^{-\frac{1}{3}}$ and we get by means of Lemmata \ref{lem:intbound} and \ref{lem:intbound2} with $F(y)=(\pm\omega'(y)+\alpha)e^{it(\pm \omega(y)+\alpha y)}, G(y)=g(ty)$ and $H(y)=[(\pm\omega'(y)+\alpha) (y^2-1)^{\frac{1}{4}}y ]^{-1} $ a satisfying contribution of $t^{-1}$.
So from now on we can assume $\alpha >0$, $\alpha \ge k t^{-\frac{1}{3}}$, for some small constant $k$, and the sign being `$-$'. We treat first the case where $\alpha<1$, where we make use of a Taylor expansion around $y_0$. We split up the integral \eqref{eq:holomain} into three parts $\CI_1,\CI_2,\CI_3$ corresponding to the intervals $[1+t^{-\frac{2}{3}},y_0-A],[y_0-A,y_0+A],[y_0+A,\infty]$ respectively. For $\CI_1$ and $\CI_3$ we again make use of Lemmata \ref{lem:intbound} and \ref{lem:intbound2} with $F(y)=(\omega'(y)-\alpha)e^{it(\omega(y)-\alpha y)}, G(y)=g(ty)$ and $H(y)=[(\omega'(y)-\alpha) (y^2-1)^{\frac{1}{4}}y ]^{-1} $. Thus we need lower bounds on 
$$
R(x)=\sqrt{x^2-1}-\alpha x \text{ and } (x^2-1)^{\frac{1}{4}}.
$$
We have
$$
R'(x)=\frac{x}{\sqrt{x^2-1}}-\alpha \text{ and } R''(x)=-\frac{1}{(x^2-1)^{\frac{3}{2}}}.
$$
We have that $R'(x)$ is decreasing and positive and hence $R(x)$ is increasing with a zero at $y_0$. Furthermore we have $R''(x)$ is increasing and negative. We conclude
$$\begin{aligned}
R(y_0+A) &\ge R(y_0)+F'(y_0) \cdot A + R''(y_0) \cdot \frac{A^2}{2}\\
&= \frac{1-\alpha^2}{\alpha} \cdot A - \left( \frac{1-\alpha^2}{\alpha^2} \right)^{\frac{3}{2}} \cdot \frac{A^2}{2} \\
&= \frac{1-\alpha^2}{\alpha} \cdot A \cdot \left(1- \frac{(1-\alpha^2)^{\frac{1}{2}}}{\alpha^2} \cdot \frac{A}{2} \right) \\
&\gg \frac{1-\alpha^2}{\alpha} \cdot A,
\end{aligned}$$
for $A \le \alpha^2(1-\alpha^2)^{-\frac{1}{2}}$. We also have
$$\begin{aligned}
-R(y_0-A) &\ge -R(y_0)+R'(y_0) A \\
& \gg \frac{1-\alpha^2}{\alpha} \cdot A.
\end{aligned}$$
For the second factor we have
$$
((y_0+A)^2-1)^{\frac{1}{4}} \ge \left( \frac{\alpha^2}{1-\alpha^2} \right)^{\frac{1}{4}}
$$
and
$$
((y_0-A)^2-1)^{\frac{1}{4}} \ge \left( \frac{\alpha^2}{1-\alpha^2} - \frac{2A}{(1-\alpha^2)^{\frac{1}{2}}} \right)^{\frac{1}{4}} \gg \left( \frac{\alpha^2}{1-\alpha^2} \right)^{\frac{1}{4}}
$$
for $A \le \frac{1}{4} \alpha^2(1-\alpha^2)^{-\frac{1}{2}}$. Thus for $A \le \frac{1}{4}\alpha^2(1-\alpha^2)^{-\frac{1}{2}}$ we find that the contribution from $\CI_3$ is at most
$$
t^{-\frac{3}{2}} \frac{1}{\left( \frac{1-\alpha^2}{\alpha} \right)A \cdot \left( \frac{\alpha^2}{1-\alpha^2} \right)^{\frac{1}{4}}} \ll t^{-\frac{3}{2}} \frac{\alpha^{\frac{1}{2}}}{(1-\alpha^2)^{\frac{3}{4}}A}.
$$
We claim that $-R(x)(x^2-1)^{\frac{1}{4}}$ increases first and then decreases in $[1,y_0]$. For this it suffices to prove that its derivative has exactly one zero in that interval and is positive at $1+\epsilon$. Note that since our function is zero at the endpoints we have by Rolle's Theorem that there is at least a zero of the derivative. The derivative is
$$
\frac{3\alpha x^2-3x(x^2-1)^{\frac{1}{2}}-2\alpha}{2(x^2-1)^{\frac{3}{4}}},
$$
which is clearly positive at $1+\epsilon$. Assume now that we have two zeros $y_1,y_2$ in $[1,y_0]$. They both satisfy the equation
$$
3\alpha x^2-3x(x^2-1)^{\frac{1}{2}}-2\alpha =0 \Rightarrow 9(1-\alpha^2)x^4+(12\alpha^2-9)x^2-4\alpha^2=0.
$$
Now by Vieta's formula we have
$$
2 \le y_1^2+y_2^2 = \frac{9-12\alpha^2}{9(1-\alpha^2)}=\frac{4}{3}-\frac{1}{3(1-\alpha^2)}\le \frac{4}{3}
$$
and thus a contradiction. With this information we conclude that if $\alpha \ge K t^{-\frac{1}{3}}$, for some large constant $K$, we have that the contribution from $\CI_1$ is at most
$$
\max \left\{t^{-1}, t^{-\frac{3}{2}} \frac{\alpha^{\frac{1}{2}}}{(1-\alpha^2)^{\frac{3}{4}}A} \right\}.
$$
Further more we estimate the integral over $\CI_2$ trivially and get the bound
$$
t^{-\frac{1}{2}} A \frac{(1-\alpha^2)^{\frac{3}{4}}}{\alpha^{\frac{1}{2}}}.
$$
Choosing $A=t^{-\frac{1}{2}} \alpha^{\frac{1}{2}}(1-\alpha^2)^{-\frac{1}{2}}$, which we are allowed for $K$ large enough we get that \eqref{eq:holomain} is bounded by
$$
t^{-1} (1-|\alpha|)^{-\frac{1}{4}}.
$$
We are left to deal with the case $\alpha \asymp t^{-\frac{1}{3}}$. In this case we elongate the interval $\CI_2$ to $[1+t^{-\frac{2}{3}},y_0+A]$ and estimate trivially again. Letting $A=\frac{1}{4}\alpha^2(1-\alpha^2)^{-\frac{1}{2}}$ we find that in this case one also has a bound of $t^{-1}$ for $\CI_2, \CI_3$. This proves the first half of \eqref{eq:holosplitsmall4}.

%
Let us assume now that $\alpha \ge \frac{2 \sqrt{2}}{3}$ such that $\alpha$ is close to $1$ and $y_0 \ge 3$. Assume $2X/t \le \frac{y_0}{2}$, in this case the integral over $\CI_2$ and $\CI_3$ are $0$, furthermore we have
$$\begin{aligned}
\min_{\substack{x \in [1+t^{-\frac{2}{3}},y_0/2] \\ x \in \frac{1}{t}\supp g}} -R(x)(x^2-1)^{\frac{1}{4}} &= \min_{\substack{x \in [1+t^{-\frac{2}{3}},y_0/2] \\ x \in \frac{1}{t}\supp g}} \frac{1-(1-\alpha^2)x^2}{\alpha x +\sqrt{x^2-1}}(x^2-1)^{\frac{1}{4}} \\
&\gg \min \left\{t^{-\frac{1}{6}},\left( \frac{X}{t} \right)^{-\frac{1}{2}}\right\},
\end{aligned}$$
thus the contribution from $\CI_1$ is bounded by
$$
t^{-\frac{3}{2}} \left( t^{\frac{1}{6}}+ \left( \frac{X}{t} \right)^{\frac{1}{2}} \right).
$$
Similarly for $\frac{1}{3}X/t \ge 2 y_0$ we have that the integral over $\CI_1$ and $\CI_2$ are $0$, and furthermore
$$\begin{aligned}
\min_{\substack{x \in [2y_0,\infty) \\ x \in \frac{1}{t}\supp g}} R(x)(x^2-1)^{\frac{1}{4}} &= \min_{\substack{x \in [2y_0,\infty) \\ x \in \frac{1}{t}\supp g}} \frac{(1-\alpha^2)x^2-1}{\alpha x +\sqrt{x^2-1}}(x^2-1)^{\frac{1}{4}} \\
&\gg \left( \frac{X}{t} \right)^{-\frac{1}{2}},
\end{aligned}$$
hence the contribution from $\CI_3$ is bounded by
$$
t^{-\frac{3}{2}} \left( \frac{X}{t} \right)^{\frac{1}{2}}.
$$
Finally when $X/t \asymp y_0$ we are able to replace $|1-|\alpha||^{-\frac{1}{4}}$ by $(X/t)^{\frac{1}{2}}$, which proves the last inequality in full for $|\alpha|<1$.

Now let us have a look at $\alpha=1$. We proceed as before only that this time the sationary phase is at infinity, thus we can directly apply Lemmata \ref{lem:intbound} and \ref{lem:intbound2} with $F(y)=(\omega'(y)-1)e^{it(\omega(y)-1 y)}, G(y)=g(ty)$ and $H(y)=[(\omega'(y)-1) (y^2-1)^{\frac{1}{4}}y ]^{-1} $. We need an upper bound on the quantity
$$
\frac{1}{(y-\sqrt{y^2-1})(y^2-1)^{\frac{1}{4}}} \text{ for } y \in [1+t^{-\frac{2}{3}},\infty) \text{ and } ty \in \supp g.
$$
This function decreases and then increases, thus it takes its maximum at the boundary. The values at the boundary are easily bounded by
$$
\max\left\{ t^{\frac{1}{6}}, \left( \frac{X}{t} \right)^{\frac{1}{2}} \right\}
$$
and therefore we find that the same upper bound as for the case $|\alpha|<1$ holds for $|\alpha|=1$.

\end{proof}

\begin{lem} Let $f$ as in the beginning of Section \ref{sec:proof} and $|\alpha|\ge 1$. Then we have for $t\ge 8$
\begin{align*}
\int_0^{\frac{t}{2}} J_t(y) f(y) \frac{dy}{y} &\ll \one_{[2X/3, \infty)}(t) \cdot t^{-\frac{1}{2}} e^{-\frac{2}{5}t}, \\
\int_{\frac{t}{2}}^{t-t^{\frac{1}{3}}} J_t(y)f(y) \frac{dy}{y} &\ll
\one_{[\frac{1}{4},\infty)}(X)\one_{[X/3,4X]}(t) \cdot t^{-1}(\log(t))^{\frac{2}{3}}, \\
\int_{t-t^{\frac{1}{3}}}^{t+t^{\frac{1}{3}}} J_t(y) f(y) \frac{dy}{y} &\ll \one_{[\frac{1}{4},\infty)}(X)\one_{[3X/16,3X]}(t) \cdot t^{-1}, \\
\int_{t+t^{\frac{1}{3}}}^{\infty} J_t(y) f(y) \frac{dy}{y} &\ll  \one_{[\frac{1}{4},\infty)}(X)\one_{[0,3X/2]}(t) \cdot t^{-1} \min \left \{ 1+||\alpha|-1|^{-\frac{1}{4}},\left( \frac{X}{t} \right)^{\frac{1}{2}}  \right\}, 
\end{align*}
where $\one_{\CI}$ is the characteristic function of the interval $\CI$.
\end{lem}
\begin{proof} We follow the argumentation as in the previous lemma. The first three inequalities follow immediately. For the last inequality we need a lower bound on
$$\begin{aligned}
\min_{\substack{y \ge 1+t^{-\frac{2}{3}}\\ y \sim X/t}} \left|(|\alpha|-\omega'(y))(y^2-1)^{\frac{1}{4}}y \right| &\gg \min_{\substack{y \ge 1+t^{-\frac{2}{3}}\\ y \sim X/t}} \left ( |\alpha|-1+\frac{y-\sqrt{y^2-1}}{y} \right)(y^2-1)^{\frac{1}{4}}y \\
&\gg \min_{\substack{y \ge 1+t^{-\frac{2}{3}}\\ y \sim X/t}} \left ( |\alpha|-1+\frac{1}{y^2} \right)(y^2-1)^{\frac{1}{4}}y.
\end{aligned}$$
If $X/t \asymp 1$, then the minimum is at least $|\alpha|t^{-\frac{1}{6}}$, which gives a contribution of $t^{-\frac{4}{3}}|\alpha|^{-1} \ll t^{-1}$, otherwise $X/t \gg 1$ in which case the minimum is at least
$$
\max \left\{ ||\alpha|-1| \left( \frac{X}{t} \right)^{\frac{3}{2}}, \left( \frac{X}{t} \right)^{-\frac{1}{2}} \right \} \gg \max \left \{ ||\alpha|-1|^{\frac{1}{4}}, \left( \frac{X}{t} \right)^{-\frac{1}{2}} \right \}
$$
giving a contribution of
$$
t^{-\frac{3}{2}} \min \left\{ ||\alpha|-1|^{-\frac{1}{4}} , \left( \frac{X}{t} \right)^{\frac{1}{2}} \right \}.
$$
\end{proof}

\begin{lem} Let $f$ be as in the beginning of Section \ref{sec:proof} and $|\alpha|\le 1$ then we have
\begin{align*}
\widehat{f}(t) &\ll \frac{1+|\log(X)|}{1+X^{\frac{1}{2}}+|1-|\alpha|^2|^{\frac{1}{2}}X}, && \forall t \in \BR, \\
\widehat{f}(t) & \ll  |t|^{-\frac{3}{2}} \left(1+\min \left\{  \left( \frac{X}{|t|} \right)^{\frac{1}{2}} , |1-|\alpha|^2|^{-1}\left( \frac{X}{|t|}\right)^{-\frac{3}{2}}\right\} \right), && \forall |t|\ge 1, \\
\widehat{f}(t) & \ll \frac{C}{T} |t|^{-\frac{5}{2}} \left( 1+ \min \left\{ \left( \frac{X}{|t|} \right)^{\frac{3}{2}} , |1-|\alpha|^2|^{-2} \left( \frac{X}{|t|} \right)^{-\frac{5}{2}}\right\}\right), && \forall |t|\ge 1. 
\end{align*}
\label{lem:nonholosmall}
\end{lem}


\begin{proof} We follow the proof of Lemma 7.1 in \cite{generalkuznetsov} and Proposition 5 in \cite{SarTsim}. To prove the first inequality we use the equation
$$\begin{aligned}
J_{2it}(x)-J_{-2it}(x) &= \frac{4i}{\pi} \sinh \pi t \int_0^{\infty} \cos(x \cosh \xi)\cos(2t\xi) d\xi. 
\end{aligned}$$
We have by partial integration
$$\begin{aligned}
\int_0^{\infty} e^{i(\pm x \cosh \xi)} \frac{f(x)}{x} dx &= \int_0^{\infty} e^{ix(\pm \cosh \xi+\alpha)} \frac{g(x)}{x} dx \\
&= \frac{i}{\pm \cosh \xi +\alpha} \int_0^{\infty} e^{ix(\pm  \cosh \xi+\alpha)} \left(\frac{g(x)}{x} \right)' dx \\
& \ll \min \left\{  1, X^{-1} |\cosh \xi \pm \alpha|^{-1}   \right \}.
\end{aligned}$$
Thus we find
$$
\widehat{f}(t) \ll \int_0^{\infty} \min \left\{  1, X^{-1} |\cosh \xi \pm \alpha|^{-1} \right\} d\xi. 
$$
Hence it suffices to bound the latter integral. It is bounded by
$$\begin{aligned}
\ll & \int_0^{1} \min \left\{  1, X^{-1} (\xi^2+1- |\alpha|)^{-1} \right\} d\xi + \int_1^{\infty} \min \left\{  1, X^{-1} e^{-\xi} \right\} d\xi \\
\ll & \int_0^{1} \min \left\{  1, X^{-1} \xi^{-2}, X^{-1} \xi^{-1}|1-|\alpha||^{-\frac{1}{2}},X^{-1} |1-|\alpha||^{-1} \right\} d\xi \\
& + \int_1^{\infty} \min \left\{  1, X^{-1} e^{-\xi} \right\} d\xi.
\end{aligned}$$
For $X \ge 1$ this is bounded by
$$
\ll \min \left \{  X^{-\frac{1}{2}}, \frac{1+\log(X)}{|1-|\alpha||^{\frac{1}{2}}X},X^{-1}|1-|\alpha||^{-1}  \right \}+ X^{-1}
$$
and for $X \le 1$ it is bounded by
$$
\ll_{\epsilon} 1+|\log(X)|.
$$
The first inequality follows immediately.\\

The final two inequalities require some more work. Note that $\widehat{f}(t)$ is even in $t$, thus we can restrict ourselves to $t \ge 1$. We make the substitution $x \to 2t x$ in the definition of $\widehat{f}(t)$
$$\begin{aligned}
\widehat{f}(t) &= \frac{i}{\sinh \pi t} \int_0^{\infty} \frac{J_{2it}(2tx)-J_{-2it}(2tx)}{2}f(2tx) \frac{dx}{x}
\end{aligned}$$
and use the uniform asymptotic expansion of the function $G_{i\nu}(\nu s)$ from \cite{Du90} pages 1009-1010 with $n=0$.
$$\begin{aligned}
G_{2it}(2tx) =& \frac{1}{\sinh(\pi t)} \frac{J_{2it}(2tx)-J_{-2ti}(2tx)}{2i} \\
=& \left( \frac{1}{\pi t} \right)^{\frac{1}{2}} (1+x^2)^{-\frac{1}{4}} \Biggl[ \sin(2t \omega(x)-\tfrac{\pi}{4})-\cos(2t\omega(x)-\tfrac{\pi}{4}) \frac{3(1+x^2)^{-\frac{1}{2}}-5(1+x^2)^{-\frac{3}{2}}}{48t} \\
&+ \frac{1}{2i}\left( e^{-i\frac{\pi}{4}} \CE_{1,1}(2t,\omega(x))-e^{i\frac{\pi}{4}}  \CE_{1,2}(2t,\omega(x)) \right) \Biggr] \\
\end{aligned}$$
here
$$
\omega(x) = \sqrt{1+x^2} +  \log \left( \frac{x}{1+\sqrt{1+x^2}} \right)
$$
and the error terms satisfy
$$
\CE_{1,1}(2t,\omega(x)), \ \CE_{1,2}(2t,\omega(x)) \ll |t|^{-\frac{5}{2}} \exp(O(|t|^{-1})).
$$
Let us first deal with the error term. The contribution of the error term is bounded by
$$
t^{-\frac{5}{2}} \int_0^{\infty} |f(2tx)| \frac{dx}{x} \ll t^{-\frac{5}{2}} \ll \min\left\{|t|^{-\frac{3}{2}},\frac{C}{T}|t|^{-\frac{5}{2}}\right\}.
$$
For the remaining summands we have to deal with integrals of the type
$$
t^{-\frac{1}{2}}\int_0^{\infty} \frac{e^{\pm 2i t \omega(x)}}{(1+x^2)^{\frac{1}{4}+\beta}} f(2tx) \frac{dx}{x}=t^{-\frac{1}{2}}\int_0^{\infty} \frac{e^{2i t (\pm\omega(x)+\alpha x)}}{(1+x^2)^{\frac{1}{4}+\beta}} g(2tx) \frac{dx}{x},
$$
with $\beta \in \{0,\frac{1}{2},\frac{3}{2}\}$. We rewrite the above as
\begin{equation}
\frac{1}{2}t^{-\frac{3}{2}} \int_0^{\infty} \left( e^{2it(\pm \omega(x)+\alpha)} 2t (\pm \omega'(x)+\alpha) \right)  \frac{g(2tx)}{x(\pm \omega'(x)+\alpha)(1+x^2)^{\frac{1}{4}+\beta}} dx.
\label{eq:beforelem}
\end{equation}
Since
$$
\omega'(x)= \frac{\sqrt{1+x^2}}{x} > 1
$$
we have $\omega'(x)-|\alpha| > 0$. We apply Lemmata \ref{lem:intbound} and \ref{lem:intbound2} with $F(x)=e^{2it(\pm \omega(x)+\alpha)} 2t (\pm \omega'(x)+\alpha), G(x)=g(2tx)$ and $H(x)=[x(\pm \omega'(x)+\alpha)(1+x^2)^{\frac{1}{4}+\beta}]^{-1}$. 
Moreover we have
$$\begin{aligned}
\min_{x \sim \frac{X}{t}} \left|x(\pm \omega'(x)+\alpha)(1+x^2)^{\frac{1}{4}+\beta} \right| &\gg\min_{x \sim \frac{X}{t}} \left| x \left( \frac{1}{x\sqrt{1+x^2}}+1-|\alpha| \right) (1+x^2)^{\frac{1}{4}} \right| \\
& \gg \min_{x \sim \frac{X}{t}} \max \left \{ (1+x^2)^{-\frac{1}{4}}, (1-|\alpha|) x (1+x^2)^{\frac{1}{4}}   \right \}.
\end{aligned}$$
For $x \ll 1$ we see that the function is bounded below by $1$. If $x \gg 1$ then the function is bounded by below by
$$
\max \left \{ \left(\frac{X}{t}\right)^{-\frac{1}{2}}, |1-|\alpha|| \left( \frac{X}{t} \right)^{\frac{3}{2}} \right \}
$$
Therefore the integral \eqref{eq:beforelem} is bounded by
$$
t^{-\frac{3}{2}} \left(1+ \min \left\{ \left( \frac{X}{t} \right)^{\frac{1}{2} }, |1-|\alpha||^{-1} \left( \frac{X}{t} \right)^{-\frac{3}{2}} \right\} \right)  .
$$
This yields the second inequality. For the third inequality we proceed from \eqref{eq:beforelem} with integration by parts. We have to deal with $4$ new integrals
$$\begin{aligned}
\CI_1 &= t^{-\frac{5}{2}} \int_0^{\infty} \left(e^{2it(\pm\omega(x)+\alpha)} 2t (\pm\omega'(x)+\alpha) \right) \frac{g(2tx)}{x^2(\pm\omega'(x)+\alpha)^2(1+x^2)^{\frac{1}{4}+\beta}} dx, \\
\CI_2 &= t^{-\frac{5}{2}} \int_0^{\infty} \left(e^{2it(\pm\omega(x)+\alpha)} 2t (\pm\omega'(x)+\alpha) \right) \frac{g(2tx)(\pm\omega''(x)x^2)}{x^3(\pm\omega'(x)+\alpha)^3(1+x^2)^{\frac{1}{4}+\beta}} dx, \\
\CI_3 &= t^{-\frac{5}{2}} \int_0^{\infty} \left(e^{2it(\pm\omega(x)+\alpha)} 2t (\pm\omega'(x)+\alpha) \right) \frac{g(2tx)x^2}{x^2(\pm\omega'(x)+\alpha)^2(1+x^2)^{\frac{5}{4}+\beta}} dx, \\
\CI_4 &= t^{-\frac{5}{2}} \int_0^{\infty} \left(e^{2it(\pm\omega(x)+\alpha)} 2t (\pm\omega'(x)+\alpha) \right) \frac{tx \cdot g'(2tx)}{x^2(\pm\omega'(x)+\alpha)^2(1+x^2)^{\frac{1}{4}+\beta}} dx.
\end{aligned}$$
Proceeding as before we find
$$\begin{aligned}
\CI_1 &\ll t^{-\frac{5}{2}} \left( 1+ \min \left \{ \left( \frac{X}{t} \right)^{\frac{3}{2}},|1-|\alpha||^{-2} \left( \frac{X}{t} \right)^{-\frac{5}{2}} \right \} \right) , \\
\CI_2 &\ll t^{-\frac{5}{2}} \left(1+ \min \left \{ \left( \frac{X}{t} \right)^{\frac{3}{2}},|1-|\alpha||^{-3} \left( \frac{X}{t} \right)^{-\frac{9}{2}} \right \} \right), \\
\CI_3 &\ll t^{-\frac{5}{2}}  \left(1+ \min \left \{ \left( \frac{X}{t} \right)^{\frac{3}{2}},|1-|\alpha||^{-2} \left( \frac{X}{t} \right)^{-\frac{5}{2}} \right \} \right), \\
\CI_4 &\ll \frac{C}{T} t^{-\frac{5}{2}} \left( 1+ \min \left \{  \left( \frac{X}{t} \right)^{\frac{3}{2}},|1-|\alpha||^{-2} \left( \frac{X}{t} \right)^{-\frac{5}{2}} \right \} \right).
\end{aligned}$$
We conclude the third inequality from this.
\end{proof}

\begin{lem} Let $f$ be as in the beginning of Section \ref{sec:proof} and $|\alpha|\ge 1$ then we have
\begin{align*}
\widehat{f}(t) &\ll  \frac{1+|\log(X)|+\log(|\alpha|)}{1+X^{\frac{1}{2}}+||\alpha|^2-1|^{\frac{1}{2}}X}, && \forall t \in \BR. 
\end{align*}
When $|t| \nin \left[\tfrac{1}{12}||\alpha|^2-1|^{\frac{1}{2}}X,2||\alpha|^2-1|^{\frac{1}{2}}X\right]$ and $|t| \ge 1$ we can do better and find in that case
\begin{align*}
\widehat{f}(t) & \ll |t|^{-\frac{3}{2}} \left(1+ \min\left\{\left( \frac{X}{|t|} \right)^{\frac{1}{2}},||\alpha|^2-1|^{-1} \left( \frac{X}{|t|} \right)^{-\frac{3}{2}} \right\} \right) , \\
\widehat{f}(t) & \ll \frac{C}{T} |t|^{-\frac{5}{2}} \left(1+ \min\left\{\left( \frac{X}{|t|} \right)^{\frac{3}{2}},||\alpha|^2-1|^{-2} \left( \frac{X}{|t|} \right)^{-\frac{5}{2}} \right\} \right). 
\end{align*}
\end{lem}

\begin{proof} We follow the proof of the previous lemma which leads us to estimate:
$$
\widehat{f}(t) \ll \int_0^{\infty} \min \left\{  1, X^{-1} |\cosh \xi - |\alpha||^{-1} \right\} d\xi. 
$$
Set $\cosh(\phi)=|\alpha|$ and note that we have $e^{\phi} \asymp |\alpha|$ and $\log(|\alpha|) \le \phi \le 1+ \log(|\alpha|) $ for $|\alpha| \ge 1$. This leads to
$$
\widehat{f}(t) \ll \int_0^{\infty} \min \left\{  1, X^{-1} \sinh\left(\frac{\xi+\phi}{2}\right)^{-1}\sinh\left( \frac{|\xi-\phi|}{2} \right)^{-1} \right\} d\xi.
$$
Thus it suffices to bound the latter integral. We split up the region of integration into three parts $\CI_1,\CI_2$ and $\CI_3$, where we restrict ourselves to $|\xi-\phi|\ge 1$, $|\xi-\phi|\le 1 \, \wedge \, \xi+\phi \ge 1$ and $|\xi-\phi|\le 1 \, \wedge \, \xi+\phi \le 1$, respectively. For $X \ge 1$ we have

$$\begin{aligned}
\CI_1 &\ll \int_0^{\infty} \min\left \{ 1, X^{-1}e^{-\max\{\phi,\xi\}}  \right\} d\xi \\
& \ll \int_0^{\phi} \frac{e^{-\phi}}{X} d\xi + \int_{\phi}^{\infty} \frac{e^{-\xi}}{X} d\xi\\
& \ll_{\epsilon} \frac{1+\log(|\alpha|)}{|\alpha| X},
\end{aligned}$$

$$\begin{aligned}
\CI_2 &\ll \int_{\max\{0,\phi-1\}}^{\phi+1} \min \left \{ 1, X^{-1} e^{-\frac{\xi+\phi}{2}} |\xi-\phi|^{-1}   \right\} d \xi \\
& \ll \int_{-1}^1 \min \left \{ 1, X^{-1} e^{-\phi} |\psi|^{-1}   \right\} d \psi \\
& \ll  \int_0^{\frac{1}{|\alpha|X}} d\psi + \int_{\frac{1}{|\alpha|X}}^1 \frac{1}{|\alpha|X \psi} d\psi  \\
& \ll \frac{1+\log(|\alpha|X)}{|\alpha| X},
\end{aligned}$$

$$\begin{aligned}
\CI_3 &\ll \int_{\max\{0,\phi-1\}}^{1-\phi} \min\left\{1,X^{-1}|\xi^2-\phi^2|^{-1}\right\} d\xi \\
& \ll \int_{\max\{-1,-\phi\}}^{1-2\phi} \min\left \{1, X^{-1}\phi^{-1} |\psi|^{-1} , X^{-1} |\psi|^{-2} \right \} d\psi \\
& \ll \one_{[0,1]}(\phi) \min \left\{1, \frac{1+\log^{+}(X\phi)}{X\phi}, X^{-\frac{1}{2}} \right\} \\
& \ll \one_{[0,1]}(\phi) \min \left\{1, \frac{1+\log(X)}{||\alpha|-1|^{\frac{1}{2}}X}, X^{-\frac{1}{2}} \right\}.
\end{aligned}$$
For $X \le 1$ we have

$$\begin{aligned}
\CI_1 &\ll \int_0^{\infty} \min\left \{ 1, X^{-1}e^{-\max\{\phi,\xi\}}  \right\} d\xi \\
& \ll \int_{0}^{\max\{\phi,-\log(X)\}} \min\left\{1 , \frac{e^{-\phi}}{X}\right\} d\xi + \int_{\max\{\phi,-\log(X)\}}^{\infty} \frac{e^{-\xi}}{X}  d\xi\\
& \ll_{\epsilon} \frac{1+\log(|\alpha|)+|\log(X)|}{1+|\alpha|X}+\frac{1}{X} \min \left \{ |\alpha|^{-1},X \right \} \\
& \ll_{\epsilon} \frac{1+\log(|\alpha|)+|\log(X)|}{1+|\alpha|X},
\end{aligned}$$

$$\begin{aligned}
\CI_2 &\ll \int_{\max\{0,\phi-1\}}^{\phi+1} \min \left \{ 1, X^{-1} e^{-\frac{\xi+\phi}{2}} |\xi-\phi|^{-1}   \right\} d \xi \\
& \ll \int_{-1}^1 \min \left \{ 1, X^{-1} e^{-\phi} |\psi|^{-1}   \right\} d \psi \\
& \ll \min \left \{ 1, \frac{1+\log^{+}(|\alpha|X)}{|\alpha|X} \right \},
\end{aligned}$$
$$\begin{aligned}
\CI_3 & \ll \int_{\max\{0,\phi-1\}}^{1-\phi} \min\left\{1,X^{-1}|\xi^2-\phi^2|^{-1}\right\} d\xi \\
& \ll \one_{[0,1]}(\phi).
\end{aligned}$$
This completes the case $X\le 1$.

For the second inequality we proceed as in Lemma \ref{lem:nonholosmall} and have to consider the integral
$$
t^{-\frac{1}{2}}\int_0^{\infty} \frac{e^{2i t (\pm\omega(x)+\alpha x)}}{(1+x^2)^{\frac{1}{4}+\beta}} g(2tx) \frac{dx}{x}.
$$
We would pick up a stationary phase at $x_0=(\alpha^2-1)^{-\frac{1}{2}}$, however we have $x \in [\frac{1}{6} \frac{X}{t},\frac{X}{t}]$ which does not intersect $[\frac{1}{2}x_0, 2 x_0]$. Thus we split up the integral into two parts $\CI_1$ and $\CI_2$ corresponding to the intervals $[0,\frac{1}{2}x_0]$ and $[2x_0,\infty)$. Without loss of generality let $\alpha \ge 1$. 
Assume first that $X/t \le 1$. In this case we have by Lemmata \ref{lem:intbound} and \ref{lem:intbound2} with the choice $F(x)=(\pm \omega'(x)+\alpha)e^{2it(\pm\omega(x)+\alpha)}, G(x)=g(2tx)$ and $H(x)=[(1+x^2)^{\frac{1}{4}+\beta} (\sqrt{1+x^2}-\alpha x)]^{-1}$,
$$\begin{aligned}
\CI_1 & \ll t^{-\frac{3}{2}} \frac{1}{\displaystyle \min_{x \in [0,\frac{1}{2}x_0]\cap [\frac{1}{6}\frac{X}{t},\frac{X}{t}]} \sqrt{1+x^2}-\alpha x } ,\\
\CI_2 & \ll t^{-\frac{3}{2}} \frac{1}{\displaystyle \min_{x \in [2x_0,\infty)\cap [\frac{1}{6}\frac{X}{t},\frac{X}{t}]} \alpha x-\sqrt{1+x^2} }.
\end{aligned}$$
The allowed range for $t$ leaves us with two cases, either $x_0 \ge 2\frac{X}{t}$ or $x_0 \le \frac{1}{12}\frac{X}{t}$. If $x_0 \ge 2\frac{X}{t}$, then we the integral over $\CI_2$ is $0$, and
$$
\sqrt{1+x^2}-\alpha x= \frac{1-x^2(\alpha^2-1)}{\sqrt{1+x^2}+\alpha x} \gg 1, \text{ for } x\le \frac{1}{2}x_0 \text{ and } x  \le 1.
$$
Thus we get a total bound of $t^{-\frac{3}{2}}$. Similarly if $x_0 \le \frac{1}{12}\frac{X}{t}$ we have that the integral over $\CI_1$ is $0$, and
$$
\alpha x-\sqrt{1+x^2}=\frac{x^2(\alpha^2-1)-1}{\sqrt{1+x^2}+\alpha x} \gg \frac{1}{\alpha x}, \text{ for } x\ge 2x_0 \text{ and } x  \le 1.
$$
Note that for $x \le 1$ we also have $\alpha x-\sqrt{1+x^2}\ge \alpha x-\sqrt{2}$ and hence
$$
\alpha x-\sqrt{1+x^2} \gg \alpha x+ \frac{1}{\alpha x} \text{ for } x\ge 2x_0 \text{ and } x  \le 1.
$$
This yields a total bound of $t^{-\frac{3}{2}}$. 

Assume now $\frac{X}{t} \ge 1$. In this case we have
$$\begin{aligned}
\CI_1 & \ll t^{-\frac{3}{2}} \frac{1}{\displaystyle \min_{x \in [0,\frac{1}{2}x_0]\cap [\frac{1}{6}\frac{X}{t},\frac{X}{t}]} \left(\sqrt{1+x^2}-\alpha x\right)x^{\frac{1}{2}} } ,\\
\CI_2 & \ll t^{-\frac{3}{2}} \frac{1}{\displaystyle \min_{x \in [2x_0,\infty)\cap [\frac{1}{6}\frac{X}{t},\frac{X}{t}]} \left(\alpha x-\sqrt{1+x^2}\right)x^{\frac{1}{2}} }.
\end{aligned}$$
If $x_0 \ge 2 \frac{X}{t}$, then we have that the integral over $\CI_2$ is $0$, and
$$
\left(\sqrt{1+x^2}-\alpha x\right)x^{\frac{1}{2}} = \frac{1-x^2(\alpha^2-1)}{\sqrt{1+x^2}+\alpha x} x^{\frac{1}{2}} \gg x^{-\frac{1}{2}}, \text{ for } x\le \frac{1}{2}x_0 \text{ and } x  \ge \frac{1}{12}.
$$
Thus we get a total bound of $t^{-\frac{3}{2}} \left(\frac{X}{t}\right)^{\frac{1}{2}} \ll t^{-\frac{3}{2}} |\alpha^2-1|^{-1}\left( \frac{X}{t}\right)^{-\frac{3}{2}}$. Similarly if $x_0 \le \frac{1}{12}\frac{X}{t}$ we have that the integral over $\CI_1$ is $0$, and
$$
\left(\alpha x-\sqrt{1+x^2}\right) = \frac{x^2(\alpha^2-1)-1}{\sqrt{1+x^2}+\alpha x} \ge \frac{3}{8} \frac{x^2(\alpha^2-1)}{\alpha x} \gg \frac{1}{\alpha x}, \text{ for } x\ge 2x_0 \text{ and } x  \ge \frac{1}{6}.
$$
This yields a total bound of
$$
t^{-\frac{3}{2}} \cdot \min \left\{ \alpha \left( \frac{X}{t} \right)^{\frac{1}{2}}, \frac{\alpha}{\alpha^2-1} \left( \frac{X}{t} \right)^{-\frac{3}{2}} \right\} \ll t^{-\frac{3}{2}} \left( 1+ \min \left\{ \left( \frac{X}{t} \right)^{\frac{1}{2}}, \frac{1}{\alpha^2-1} \left( \frac{X}{t} \right)^{-\frac{3}{2}} \right\} \right),
$$
since $\frac{X}{t}\ge 1$.
This proves the second inequality.

For the third inequality we integrate once by parts. We then have to consider the integrals
$$\begin{aligned}
\CI_4 &= t^{-\frac{5}{2}} \int_0^{\infty} \left(e^{2it(\pm\omega(x)+\alpha)} 2t (\pm\omega'(x)+\alpha) \right) \frac{g(2tx)}{x^2(\pm\omega'(x)+\alpha)^2(1+x^2)^{\frac{1}{4}+\beta}} dx, \\
\CI_5 &= t^{-\frac{5}{2}} \int_0^{\infty} \left(e^{2it(\pm\omega(x)+\alpha)} 2t (\pm\omega'(x)+\alpha) \right) \frac{g(2tx)(\pm\omega''(x)x^2)}{x^3(\pm\omega'(x)+\alpha)^3(1+x^2)^{\frac{1}{4}+\beta}} dx, \\
\CI_6 &= t^{-\frac{5}{2}} \int_0^{\infty} \left(e^{2it(\pm\omega(x)+\alpha)} 2t (\pm\omega'(x)+\alpha) \right) \frac{g(2tx)x^2}{x^2(\pm\omega'(x)+\alpha)^2(1+x^2)^{\frac{5}{4}+\beta}} dx, \\
\CI_7 &= t^{-\frac{5}{2}} \int_0^{\infty} \left(e^{2it(\pm\omega(x)+\alpha)} 2t (\pm\omega'(x)+\alpha) \right) \frac{tx \cdot g'(2tx)}{x^2(\pm\omega'(x)+\alpha)^2(1+x^2)^{\frac{1}{4}+\beta}} dx.
\end{aligned}$$
By similar means as before we have that
$$\begin{aligned}
\CI_4 &\ll t^{-\frac{5}{2}}\left(1+ \min \left \{ \left( \frac{X}{t} \right)^{\frac{3}{2}},||\alpha|^2-1|^{-2} \left( \frac{X}{t} \right)^{-\frac{5}{2}} \right \} \right) , \\
\CI_5 &\ll  t^{-\frac{5}{2}} \left(1+ \min \left \{ \left( \frac{X}{t} \right)^{\frac{3}{2}},||\alpha|^2-1|^{-3} \left( \frac{X}{t} \right)^{-\frac{9}{2}} \right \}\right), \\
\CI_6 &\ll t^{-\frac{5}{2}} \left( 1+ \min \left \{ \left( \frac{X}{t} \right)^{\frac{3}{2}},||\alpha|^2-1|^{-2} \left( \frac{X}{t} \right)^{-\frac{5}{2}} \right \} \right), \\
\CI_7 &\ll \frac{C}{T} t^{-\frac{5}{2}} \left(1+ \min \left \{ \left( \frac{X}{t} \right)^{\frac{3}{2}},||\alpha|^2-1|^{-2} \left( \frac{X}{t} \right)^{-\frac{5}{2}} \right \} \right).
\end{aligned}$$
We conclude the last inequality from this.

\end{proof}

\begin{lem} Let $f$ be as in the beginning of Section \ref{sec:proof}. For $0 \le t \le \frac{1}{4}-\delta$ we have the following expansion
$$
\widehat{f}(it)=-\frac{1}{2}\int_{\frac{X}{2}}^{X} Y_{2t}(x)e^{i\alpha x}\frac{dx}{x} + O_{\epsilon, \delta}\left(1+ \frac{T}{C}X^{-2t-\epsilon} \right) 
$$
\end{lem}
\begin{proof} We have
$$\begin{aligned}
\widehat{f}(it)&=\frac{1}{\sin (2\pi t)} \int_0^{\infty} \frac{J_{-2t}(x)-J_{2t}(x)}{2}f(x)\frac{dx}{x} \\
&= -\frac{1}{2}\int_0^{\infty} \left[ \frac{J_{2t}(x)\cos(2\pi t)-J_{-2t}(x)}{\sin(2 \pi t)}+ \frac{J_{2t}(x)-J_{2t}(x)\cos(2\pi t)}{\sin(2 \pi t)} \right]f(x)\frac{dx}{x} \\
&= -\frac{1}{2}\int_0^{\infty} \left[ Y_{2t}(x)+J_{2t}(x)\tan(\pi t) \right]f(x)\frac{dx}{x}.
\end{aligned}$$
Now we have
$$
\int_0^{\infty} J_{2t}(x)\tan(\pi t)f(x)\frac{dx}{x} \ll \int_{0}^{\infty} \min \left \{ x^{2t}, x^{-\frac{1}{2}} \right \} \frac{g(x)}{x} dx \ll 1
$$
and
$$
\left(\int_{\frac{2 \pi \sqrt{mn}}{(C+T)}}^{\frac{X}{2}}+\int_{X}^{\frac{4 \pi \sqrt{mn}}{(C-T)}} \right) Y_{2t}(x) f(x) \frac{dx}{x} \ll \frac{T}{C} \sup_{x \sim X} |Y_{2t}(x)|.
$$
The following inequality will imply the result
$$
|Y_{2t}(x)| \ll_{\epsilon} \begin{cases} x^{-2t-\epsilon}, & \text{if } x \le 1, \\ x^{-\frac{1}{2}}, &\text{if }x \ge 1. \end{cases}
$$
The range $x \ge 1$ can be found in \cite[Appendix B.35]{IwaniecSpectral} and for the range $x \le 1$ we make use of the following integral representation \cite[page 170]{ToBF}:
$$
Y_{2t}(x)=- \frac{2 (\frac{x}{2})^{-2t}}{\sqrt{\pi} \Gamma(\frac{1}{2}-2t)} \int_1^{\infty} \frac{\cos(xy)}{(y^2-1)^{2t+\frac{1}{2}}} dy.
$$
The integral from $1$ to $\frac{1}{x}$ is bounded by
$$\begin{aligned}
\int_1^2\frac{1}{(y-1)^{1-2\delta}}dy+&\int_2^{\max\{2,\frac{1}{x}\}} \frac{1}{(y^2-1)^{\frac{1}{2}}}dy \\
&= \frac{1}{2\delta}(y-1)^{2\delta} \Bigg |_{y=1}^{2} + \log \left( \sqrt{y^2-1}+y \right) \Bigg |_{y=2}^{\max\{2,\frac{1}{x}\}} \ll_{\epsilon,\delta} x^{-\epsilon}
\end{aligned}$$
and the remaining integral is bounded by $O(1)$, by Lemma \ref{lem:intbound} with $F(y)=\cos(xy)$ and $G(y)=(y^2-1)^{2t+\frac{1}{2}}$.

\end{proof}

\bibliography{Bibliography}
\end{document}